\documentclass[11pt]{amsart}
\usepackage{geometry, color}
\usepackage{graphicx}
\usepackage{amssymb}
\usepackage{amsmath}
\usepackage{amsthm}
\newtheorem{definition}{Definition}
\newtheorem{lemma}{Lemma}
\newtheorem{theorem}{Theorem}

\newtheorem{proposition}{Proposition}
\newtheorem{remark}{Remark}
\newcommand{\R}{\mathbb R}
\numberwithin{equation}{section}

\numberwithin{equation}{section}
\numberwithin{lemma}{section}
\numberwithin{theorem}{section}
\numberwithin{proposition}{section}
\title[Small sphere limit of quasi-local energy]{Small sphere limit of the quasi-local energy with anti de-Sitter space reference}
\author{Po-Ning Chen}

\date{}

\begin{document}
\thanks{The author is supported by NSF grant DMS-1308164. The author would like to thank Mu-Tao Wang and Shing-Tung Yau for helpful discussion.} 

\maketitle
\begin{abstract}
In \cite{Chen-Wang-Yau3}, a new quasi-local energy is introduced for spacetimes with a non-zero cosmological constant. In this article, we study the small sphere limit of this newly defined quasi-local energy for spacetimes with a negative cosmological constant. For such spacetimes, the anti de-Sitter space is used as the reference for the quasi-local energy. Given a point $p$ in a spacetime $N$, we consider a canonical family of surfaces approaching $p$ along its future null cone and evaluate the limit of the quasi-local energy. The optimal embedding equation which identifies the critical points of the quasi-local energy is solved in order to evaluate the limit. Using the optimal embedding, we show that the limit recovers the stress-energy tensor of the matter field at $p$. For vacuum spacetimes, the quasi-local energy vanishes to a higher order. In this case, the limit of the quasi-local energy is related to the Bel--Robinson tensor at $p$.
\end{abstract}

\section{Introduction}
In general relativity, a spacetime is a  4-manifold $N$ with a Lorentzian metric $g_{\alpha\beta}$ satisfying the {\it Einstein equation}
\[R_{\alpha\beta}-\frac{R}{2}g_{\alpha\beta}+ \Lambda g_{\alpha\beta} =8\pi T_{\alpha\beta},\]
where $R_{\alpha\beta}$ and $R$ are the Ricci curvature and the scalar curvature of the metric $g_{\alpha\beta}$, respectively. The constant $ \Lambda $ is called the cosmological constant.
On the right hand side of the Einstein equation, $T_{\alpha\beta}$ is the stress-energy tensor of the matter field. For a vacuum spacetime where $T_{\alpha\beta}=0$ (which implies $R_{\alpha \beta} =\Lambda g_{\alpha \beta}$), the gravitational energy is typically measured by the {\it Bel--Robinson tensor} \cite{Bel}
\[ Q_{\mu\nu \alpha \beta} = W^{\rho \,\,\,\, \sigma}_{\,\,\,\, \mu \,\,\,\, \alpha}W_{\rho \nu \sigma \beta}+W^{\rho \,\,\,\, \sigma}_{\,\,\,\, \mu \,\,\,\, \beta}W_{\rho \nu \sigma \alpha} - \frac{1}{2}g_{\mu\nu}W_{\alpha}^{\,\,\,\,  \rho \sigma \tau}W_{\beta  \rho \sigma \tau},  \]
where $W_{\alpha \beta \gamma \delta}$ is the Weyl curvature tensor of the spacetime $N$. The stress-energy tensor and the Bel--Robinson tensor are useful in studying the global structure of the maximal development of the initial value problem in general relativity, see for example \cite{Bieri-Zipser,Christodoulou-Klainerman}.

When studying different notions of quasi-local energy, it is natural to evaluate the large sphere and the small sphere limits of the quasi-local energy to compare with the canonical measures of the gravitational energy in these situations. One expects the following \cite{Christodoulou-Yau,Penrose}:

\

1) For a family of surfaces approaching the infinity of an isolated system (the large sphere limit), the limit of the quasi-local energy recovers the total energy-momentum of the isolated system.

2) For a family of surfaces approaching a point $p$ (the small sphere limit), the limit of the quasi-local energy recovers the stress-energy tensor in spacetimes with matter fields and the Bel--Robinson tensor for vacuum spacetimes.

\

For spacetimes with $\Lambda = 0$, there are many works on evaluating the large sphere and the small sphere limits of different notions of quasi-local energy. See for example \cite{blk1,blk2,Chen-Wang-Yau1,Chen-Wang-Yau2,fst,Geroch,Horowitz-Schmidt,Huiken-Ilmanen,Kwong-Tam,Miao-Tam-Xie,Wang-Yau3,Wiygul,yu}.  The list we give here is by no means exhaustive. For a more comprehensive review of different notions of quasi-local energy and their limiting behaviors, see \cite{sz}  and the references therein. In a sequence of papers with Wang and Yau \cite{Chen-Wang-Yau1,Chen-Wang-Yau2,Wang-Yau3}, the above properties are confirmed for the Wang--Yau quasi-local energy. In particular, the small sphere limit of the Wang-Yau quasi-local energy is evaluated in \cite{Chen-Wang-Yau2} for a canonical family of surfaces approaching a point along its future null cone.  

In \cite{Chen-Wang-Yau3},  quasi-local energy and quasi-local conserved quantities are defined for spacetimes with a non-zero cosmological constant. In the same paper, the large sphere limit of the newly defined quasi-local conserved quantities is evaluated for asymptotically AdS initial data sets. It is proved that the large sphere limit of the quasi-local conserved quantities recovers the total conserved quantities for such initial data sets \cite{Abbott-Deser,Ashtekar-Magnon,Chrusciel-Herzlich,Chruscie-Nagy,Gibbons-Hull-Warner,Henneaux-T}. In this article, we evaluate the small sphere limit of the new quasi-local energy and confirm the second expected property.

The construction of the quasi-local energy is based on the Hamilton--Jacobi analysis of the gravitational action using isometric embedding of the surface into the reference space as the ground state. That is, an energy is assigned to each pair of an isometric embedding of the surface into the reference space and an observer Killing field in the reference space. For $\Lambda = 0$, the reference space for the Wang-Yau quasi-local energy is the Minkowski space. On the other hand, for $\Lambda < 0$, the reference space is the anti de-Sitter space (AdS space). The quasi-local mass is then defined to be the minimum of the assigned quasi-local energy among all possible pair. The Euler--Lagrange equation for this energy functional is referred to as the optimal embedding equation. 

To evaluate the small sphere limit of the quasi-local energy, we first study the limiting behavior of the optimal embedding equation. For the Wang--Yau quasi-local energy, the optimal embedding equation is studied in details in \cite{Chen-Wang-Wang,Chen-Wang-Yau1,Chen-Wang-Yau4,mt} which played an important role in \cite{Chen-Wang-Yau2} for evaluating the small sphere limit.  For spacetimes with a negative cosmological constant, the AdS space is used as the reference and the optimal embedding equation is more complicated for the following reasons:

\

1) While the existence of the isometric embedding is guaranteed by the work of \cite{Lin-Wang} by Lin and Wang, the isometric embedding has to be solved explicitly to evaluate the small sphere limit. For the AdS space, the static potential is coupled to the isometric embedding equation and makes it more difficult to solve explicitly.

2) The kernel for the optimal embedding is larger. In both \cite{Chen-Wang-Yau2} and this article, the optimal embedding equation for a surface in the reference space is used extensively to simplify the optimal embedding equation for the physical surface. However, due to the difference in the set of observer Killing field, the kernel of the optimal embedding equation for the AdS space is larger than that of the Minkowski space. This creates new difficulties in solving the optimal embedding equation.

\

Due to the above difficulties, for the AdS reference case, we can not recover all the general theorems in \cite{Chen-Wang-Wang,Chen-Wang-Yau1,Chen-Wang-Yau4,mt} concerning the optimal embedding equation of the Wang-Yau quasi-local energy. Nevertheless, we are able to obtain the results necessary to evaluate the small sphere limit. For a spacetime with matter fields, there is a unique choice of the leading term of the observer killing field such that the leading term of the optimal embedding equation is solvable. However, for a vacuum spacetime,  the quasi-local energy vanishes to higher order and the invertibility of the optimal embedding equation is more subtle. In fact, the leading order term of the optimal embedding equation is solvable for any choice of the observer killing field $T_0$. We will compute the qausi-local energy for each $T_0$ and the corresponding solution to the optimal embedding equation.

The structure of this article is as follows: In Section 2, we review the AdS space and its Killing fields. In Section 3, we review the quasi-local energy with reference in the AdS space. In Section 4, we describe the setting for the small sphere limit. In Section 5, we compute the expansions of the induced metric, the second fundamental forms and the connection 1-form in the small sphere limit. Using the expansions, we expend the optimal embedding equation in Section 6 and compute the non-vacuum small sphere limit of the quasi-local energy in Section 7, see Theorem \ref{thm_small_non}. The rest of the article is devoted to the small sphere limit in vacuum spacetimes. In Section 8, for each observer Killing field, we compute the leading order term of the isometric embedding solving the leading order term of the optimal embedding equation. The isometric embedding, which depends on the choice of the observer $T_0$, is denoted by $Y(T_0)$. In the next four sections, the quasi-local energy associated to the pair $(Y(T_0),T_0)$ is computed. Section 9, 10 and 11 are used to compute the three separate terms in the quasi-local energy and these results are combined in Section 12 to evaluate the limit of the quasi-local energy, see Theorem \ref{thm_small_vac}.


\section{Anti de-Sitter space and its Killing fields}
We review the AdS space and its Killing fields in this section.
Take $\R^{3,2}$ with the coordinate system $(y^0, y^1, y^2, y^3, y^4)$ and the metric
\[-(dy^4)^2+\sum_{i=1}^3(dy^i)^2-(dy^0)^2.\]
The AdS space can be identified with the hypersurface in $\R^{3,2}$
given by
\[-(y^4)^2+\sum_{i=1}^3(y^i)^2-(y^0)^2=-\frac{1}{\kappa^2}.\]

Consider the following parametrization of AdS space:
\begin{align*}
y^0&=\sqrt{\frac{1}{\kappa^2}+r^2}\sin t\\
y^1&=r\sin\theta\sin\phi\\
y^2&=r \sin\theta \cos\phi\\
y^3&=r \cos\theta\\
y^4&=\sqrt{\frac{1}{\kappa^2}+r^2}\cos t.
\end{align*}
This gives the static chart of the AdS space
\begin{align*}
-(1+\kappa^2 r^2) {d}t^2+\frac{d r^2}{(1+\kappa^2 r^2)}+r^2(d\theta^2+\sin^2\theta d\phi^2)
\end{align*}
and $V = \sqrt{1+\kappa^2 r^2}$ be the static potential of the AdS space.

The group $SO(3,2)$ leaves this hypersurface invariant and thus the isometry group of the AdS Space is $SO(3,2)$, which is $10$ dimensional. In particular,
a Killing field of the AdS space can be written as
\begin{equation}\label{Killing}
{\mathfrak K}=  A (y^0\frac{\partial}{\partial y^4}-y^4\frac{\partial}{\partial y^0} )- B_i( y^0\frac{\partial}{\partial y^i}+y^i\frac{\partial}{\partial y^0})- C_j (y^4\frac{\partial}{\partial y^j}+y^j\frac{\partial}{\partial y^4})+ D_p \epsilon_{pqr}y^q\frac{\partial}{\partial y^r}.  \end{equation}
For simplicity, we will write ${\mathfrak K}=(A,\vec{B},\vec{C},\vec{D})$ and consider $\vec{B}, \vec{C}$, and $\vec{D}$ as vectors in $\R^3$. An observer Killing field $T_0$ is a timelike hypersurface-orthogonal Killing field such that
\[  \min - \langle T_0,T_0 \rangle = 1. \]

The observer Killing fields in the AdS space are characterized in \cite[Proposition 3.1]{Chen-Hung-Wang-Yau}.
\begin{proposition}\label{observer_constraint}  A Killing field ${\mathfrak K}$ of the form \eqref{Killing}  is an observer Killing field if and only if
\[
\begin{split}
A \vec{D} = &  - \vec{B} \times \vec{C}\\
A > & \max \{ |\vec{B}|, |\vec{C}|, |\vec{D}| \}
\end{split}
\] and
\[ A^2 +  |\vec{D}|^2  -   |\vec{B}|^2-|\vec{C}|^2=\kappa^2.\]
\end{proposition}
\begin{remark}
Proposition 3.1 of \cite{Chen-Hung-Wang-Yau} states the above result for $\kappa=1$. It is straightforward to recover the result for general $\kappa$ from the proof.
\end{remark}
\begin{remark}
In particular, for an observer Killing field, we have 
\begin{equation} \label{obser_inequlaity}  A \ge \sqrt{\kappa^2+|\vec{C}|^2 } \end{equation}
\end{remark}
We will later normalize our spacetime by choosing $\kappa=1$. This corresponds to $\Lambda=-3$ in the Einstein equation.

\section{Quasi-local energy with anti de-Sitter reference}
In this section, we review the quasi-local energy with reference in the AdS space defined in \cite{Chen-Wang-Yau3}. Let $\Sigma$ be a closed embedded spacelike 2-surface in a spacetime $N$. We assume the mean curvature vector $H$ of $\Sigma$ is spacelike.  Let $J$ be the reflection of $H$ through the future outgoing light cone in the normal bundle of $\Sigma$.
The data used in the definition of the quasi-local energy is the triple $(\sigma,|H|,\alpha_H)$ on $\Sigma$ where $\sigma$ is the induced metric, $|H|$ is the norm of the mean curvature vector, and $\alpha_H$ is the connection 1-form of the normal bundle with respect to the mean curvature vector
\[ \alpha_H(\cdot )=\langle \nabla^N_{(\cdot)}   \frac{J}{|H|}, \frac{H}{|H|}   \rangle  \]
where $\nabla^N$ is the covariant derivative in $N$.

Given an isometric embedding $Y$ of $\Sigma$  into the AdS space and the observer Killing field  $\frac{\partial}{\partial t}$, let $\tau$ be the restriction of $t$ to $Y(\Sigma)$. Suppose the projection $\widehat{Y}$  of $Y(\Sigma)$ onto the static slice $t=0$ is embedded, and denote the induced metric, the second fundamental form, and the mean curvature of the image surface $\widehat{\Sigma}$ of $\widehat{Y}$  by $\hat{\sigma}_{ab}$, $\hat{h}_{ab}$, and $\widehat{H}$, respectively.  The quasi-local energy $E(\Sigma, Y, \frac{\partial}{\partial t})$  of $\Sigma$ with respect to the pair $(Y, \frac{\partial}{\partial t})$ is
\begin{equation}\label{energy_fix_chart_base}
\begin{split}
  E(\Sigma, Y,\frac{\partial}{\partial t})
= & \frac{1}{8 \pi}  \Big \{  \int V \widehat H d \widehat \Sigma -
 \int  \Big [ \sqrt{(1+V ^2| \nabla \tau|^2) |H|^2  V ^2 + div(V ^2 \nabla \tau)^2 }  \\
& \qquad -   div(V ^2 \nabla \tau)  \sinh^{-1} \frac{ div(V ^2 \nabla \tau) }{V |H|\sqrt{1+V ^2| \nabla \tau|^2} }
  - V ^2 \alpha_{H} (\nabla \tau)  \Big ] d \Sigma
\Big \},
\end{split}
\end{equation}
where $\nabla$ and $div$ are the covariant derivatives and the divergence with respect to the induced metric $\sigma$ of the surface $\Sigma$, respectively.

Let $H_0$ and $\alpha_{H_0}$ be the mean curvature vector and the connection form of $Y(\Sigma)$ in the AdS space. In terms of $H_0$ and $\alpha_{H_0}$,  we have
\begin{equation}\label{energy_fix_chart_graph}
\begin{split}
   E(\Sigma, Y,\frac{\partial}{\partial t})
=& \frac{1}{8 \pi}  \Big \{
 \int  \Big [ \sqrt{(1+V ^2| \nabla \tau|^2) |H_0|^2 V ^2 + div(V ^2 \nabla \tau)^2 }  \\
& \qquad -   div(V ^2 \nabla \tau)  \sinh^{-1} \frac{ div(V ^2 \nabla \tau) }{V |H_0|\sqrt{1+V ^2| \nabla \tau|^2} }
  - V ^2 \alpha_{H_0} (\nabla \tau)  \Big ] d \Sigma\\
& -  \int  \Big [ \sqrt{(1+V ^2| \nabla \tau|^2) |H|^2  V ^2 + div(V ^2 \nabla \tau)^2 }  \\
& \qquad -   div(V ^2 \nabla \tau)  \sinh^{-1} \frac{ div(V ^2 \nabla \tau) }{V |H|\sqrt{1+V ^2| \nabla \tau|^2} }
  - V ^2 \alpha_{H} (\nabla \tau)  \Big ] d \Sigma
\Big \}.
\end{split}
\end{equation}
While the above expressions seems to depend on the choice of the static chart, we can rewrite it purely in terms of the isometric embedding $Y$ and the observer $\frac{\partial}{\partial t}$. In fact,
\begin{equation} \label{interpret}
\begin{split}
V ^2 = &- \langle \frac{\partial}{\partial t},\frac{\partial}{\partial t} \rangle\\
V ^2  \nabla \tau=& - (\frac{\partial}{\partial t})^\top,
\end{split}
\end{equation}
where $(\frac{\partial}{\partial t})^\top$ denotes the tangential component of $\frac{\partial}{\partial t}$ to $Y(\Sigma)$. This allows us to define  $E(\Sigma,Y,T_0)$ for each pair of an isometric embedding $Y$ and an observer Killing field $T_0$ using \eqref{energy_fix_chart_base} via \eqref{interpret}.  Equivalently, we can define $E(\Sigma,Y,T_0)$  as follows:
\begin{definition}\label{energy_invariant}
The quasi-local energy $E(\Sigma, Y,T_0)$ of $\Sigma$ with respect to the pair $(Y,T_0)$ of an isometric embedding $Y$ and an observer $T_0$ is
\[
\begin{split}
   & 8 \pi E(\Sigma, Y,T_0) \\
=&
 \int_{\Sigma}  \Big [ \sqrt{  - \langle T_0^\perp,T_0^\perp \rangle |H_0|^2  + div(T_0^\top)^2 }  -   div(T_0^\top)  \sinh^{-1} \frac{ div(T_0^\top) }{|H_0|\sqrt{-  \langle T_0^\perp,T_0^\perp \rangle} }   + \alpha_{H_0} (T_0^\top)  \Big ]  d\Sigma\\
   & - \int_{\Sigma}  \Big [ \sqrt{ -  \langle T_0^\perp,T_0^\perp \rangle |H|^2  + div(T_0^\top)^2 }  -   div(T_0^\top)  \sinh^{-1} \frac{ div(T_0^\top) }{|H|\sqrt{-  \langle T_0^\perp,T_0^\perp \rangle} }   + \alpha_{H} (T_0^\top)  \Big ]d\Sigma .
\end{split}
\]
where $T_0^\perp$ is the normal part of $T_0$ to $Y(\Sigma)$.
\end{definition}

\begin{remark} \label{remark_2}
From the above formulation, it follows that the quasi-local energy $E(\Sigma, Y,T_0)$ is equivariant. Namely, that the energy is invariant if an isometry of the AdS space acts on $Y$ and $T_0$ at the same time.
\end{remark}
It is convenient to rewrite the quasi-local energy in terms of the quasi-local energy density and the quasi-local momentum density.
\begin{definition} The quasi-local energy density with respect to $(Y, T_0)$ is defined to be
\begin{equation} \label{rho} \begin{split} f &= \frac{\sqrt{|H_0|^2 +\frac{div (V ^2 \nabla \tau)^2}{V ^2+V ^4 |\nabla \tau|^2}} - \sqrt{|H|^2 +\frac{div (V ^2 \nabla \tau)^2}{V ^2+V ^4 |\nabla \tau|^2}} }{ V \sqrt{1+ V ^2|\nabla \tau|^2}}. \end{split}\end{equation}
The quasi-local momentum density  with respect to $(Y, T_0)$ is defined to be
\begin{equation}\label{j_momentum}  j = f V ^2 d \tau - d[ \sinh^{-1} (\frac{f div  (V ^2 \nabla \tau)}{|H_0||H|})]-\alpha_{H_0}  + \alpha_{H}. \end{equation}
\end{definition}
In terms of $f$ and $j$, we have
\begin{equation}\label{qlcq2}E(\Sigma, Y, T_0)=-\frac{1}{8\pi} \int_\Sigma
\left[ \langle T_0, T_0\rangle f+j(T_0^\top) \right]d\Sigma\end{equation}
The first variation of the quasi-local energy is evaluated in \cite[Theorem 5.4]{Chen-Wang-Yau3}. It will be used later in Lemma \ref{optimal_eq_leading}. For reader's convenience, the formula will be recalled in the proof of  Lemma \ref{optimal_eq_leading}.

\section{The small spheres}
We setup the small sphere limit as in \cite{Chen-Wang-Yau2}.  Let $p$ be a point in a spacetime $N$.  Let $C_p$ be the future null hypersurface generated by future null geodesics  starting at $p$. Pick any future directed timelike unit vector  $e_0$ at $p$. Using $e_0$, we normalize a null vector $L$ at $p$ by
\[  \langle L , e_0 \rangle =-1 .\]
We consider the null geodesics of the normalized $L$ and let $r$ be the affine parameter of these null geodesics. Let $\Sigma_r$ be the family of surfaces on $C_p$ defined by the level sets of the affine parameter $r$. The inward null normal $\underline L$ of $\Sigma_r$ is normalized so that
\[  \langle L ,\underline L\rangle =-1 . \]

We parametrize $\Sigma_r$ in the following way. Consider a smooth map
\begin{equation}\label{parametrization} X:S^2\times [0, \epsilon)\rightarrow N\end{equation} such that
for each fixed point in $S^2$, $X(\cdot, r), r\in [0, \epsilon)$ is a null geodesic parametrized by the affine parameter $r$, with $X(\cdot, 0)=p$
and $\frac{\partial X}{\partial r}(\cdot, 0)\in T_pN$ a null vector such that $\langle \frac{\partial X}{\partial r}(\cdot, 0), e_0\rangle=-1$ .
Let  $L=\frac{\partial X}{\partial r}$ be the null generator, $\nabla^N_L L=0$. We also choose a local coordinate system $\{u^a\}_{a=1, 2}$ on $S^2$ such that $\partial_a=\frac{\partial X}{\partial u^a}, a=1,2$ form a tangent basis to $\Sigma_r$. Let $\underline L$ be the null normal vector field along $\Sigma_r$ such that $\langle L, \underline{L}\rangle=-1$.
Denote
\[
\begin{split}
 l_{ab} = & \langle   \nabla^N_{\partial_a}  \partial_b , L\rangle    \\
 n_{ab} =& \langle   \nabla^N_{\partial_a}  \partial_b , \underline L\rangle   \\
 \eta_a  = &  \langle   \nabla^N_{L} \partial_a , \underline L \rangle
\end{split}
\]
for the second fundamental forms in the direction of $L$ and $\underline L$ and the connection 1-form in the null normal frame, respectively. We consider these as tensors on $S^2$ depending on $r$ and use the induced metric on $\Sigma_r$, $\sigma_{ab}=\langle \partial_a, \partial_b\rangle$, to raise or lower indexes. We have
\begin{equation}\label{tangent_covariant}
\begin{split}
 \nabla^N_{\partial_a} L&=-l_a^c \partial_c-\eta_a L\\
 \nabla^N_{\partial_a} \partial_b&=\gamma_{ab}^c\partial_c-l_{ab}\underline L-n_{ab}L\\
 \nabla^N_{\partial_a}\underline L&=-n_a^c\partial_c+\eta_a \underline L,
\end{split}
\end{equation} where $\gamma_{ab}^c$ are the Christoffel symbols of $\sigma_{ab}$. Let
\[
\begin{split}
\hat{l}_{ab}= &l_{ab}-\frac{1}{2}(\sigma^{cd}l_{cd})\sigma_{ab} \\
\hat{n}_{ab}=& l_{ab}-\frac{1}{2}(\sigma^{cd}l_{cd})\sigma_{ab}
\end{split}
\]
be the traceless part of $l_{ab}$ and $n_{ab}$.

The following identities for covariant derivatives are useful.
\[
\begin{split}
\nabla^N_{L}  \partial_a  =&  -l_a^c \partial_c - \eta_a L  \\
 \nabla^N_{L}  \underline L  = &  -\eta^b \partial_b .
\end{split}
\]
We consider $\sigma_{ab}, l_{ab}, n_{ab}, \eta_a$ as tensors on $S^2\times [0, \epsilon)$, or
tensors on $S^2$ that depend on the parameter $r$.  We shall see below that they have the following expansions.
\[
\sigma_{ab}=\tilde{\sigma}_{ab} r^2+O(r^3), \,\,  l_{ab}=-\tilde{\sigma}_{ab} r+O(r^2), \,\, n_{ab}=\frac{1}{2} \tilde{\sigma}_{ab} r+O(r^2), \,\,  \eta_{a}=\frac{1}{3} {\beta}_a r^2+O(r^3)
\] where $\beta_a=\lim_{r\rightarrow 0} R_{La L\underline L}$ is considered as a $(0,1)$ tensor on $S^2$, $\tilde \sigma_{ab}$ denotes the standard metric on unit $S^2$. Let $\tilde \nabla$ and $\tilde \Delta$ be the covariant derivative and the Laplacian with respect to $\tilde \sigma_{ab}$, respectively.

We shall also consider the pull-back of tensors from the null hypersurface. For example, we consider
$R(L, \cdot, L, \underline L)$ as a tensor defined on $C_p$ and take its pull-back through \eqref{parametrization}, which is then consider as a
$(0,1)$ tenors on $S^2$ that depends on $r$ (or on $S^2\times [0, \epsilon)$). We shall abuse the notations and still denote the pull-back tensor by $R_{L a L\underline L}$.
In particular, $R_{LabL }, R_{LaL\underline L}, R_{L\underline LL\underline L  }$ are considered as $r$ dependent $(0, 2)$ tensor,
$(0, 1)$ tensor, and a scalar function on $S^2$, respectively, of the following orders
\[ R_{LabL } = O( r^2), \,\, R_{LaL\underline L}= O(r) \,\,  {\rm and} \,\, R_{L\underline LL\underline L  } = O(1). \]

We first write down the expansions of $L$ and $\partial_a$. Let $x^0, x^i, i=1, 2, 3$ be a normal coordinates system at $p$ such that the original future timelike vector $e_0\in T_p N$ is $\frac{\partial}{\partial x^0}$. The parametrization \eqref{parametrization} is given by
\[X(u^a, r)=X^0(u^a, r)\frac{\partial}{\partial x^0}+X^i(u^a, r)\frac{\partial}{\partial x^i}\] with the following expansions:
\[\begin{split} X^0(u^a, r)&=r+O(r^2)\\
X^i(u^a, r)&=r \tilde{X}^i (u^a)+O(r^2), \end{split}\] where $\tilde{X}^i(u^a)$ are the three first eigenfunctions of the standard metric $\tilde{\sigma}_{ab}$ on $S^2$. For example, if we take the coordinates $u_a, a=1,2$ to be the standard spherical coordinate system $\theta, \phi$ with $\tilde{\sigma}=d\theta^2+\sin^2\theta d\phi^2$, then $\tilde{X}^1=\sin\theta\sin\phi$, $\tilde{X}^2=\sin\theta \cos\phi$, and $\tilde{X}^3=\cos\theta$.
In particular,
\begin{equation}\label{frame}\begin{split} L&=\frac{\partial X}{\partial r}=\frac{\partial}{\partial x^0}+\tilde{X}^i(u^a) \frac{\partial}{\partial x^i}+O(r)\\
\partial_a&=\frac{\partial X}{\partial u^a}=r\frac{\partial \tilde{X}^i}{\partial u^a}\frac{\partial}{\partial x^i}+O(r^2),\,\, a=1, 2.\end{split}\end{equation}

\section{The Expansion of the physical data}\label{sec_physicaldata}
In this section, we compute  the expansions of the induced metric, the second fundamental forms and the  connection 1-form of $\Sigma_r$. We compute the expansion of the geometric quantities in terms of the affine parameter $r$. In the first subsection, we state the expansion in non-vacuum spacetimes. The result is exactly the same as in Section 3.1 of \cite{Chen-Wang-Yau2}. We collect these results in the first subsection to be used later. In the second subsection, we derive the expansion in vacuum spacetimes with $\Lambda=-3$.
\subsection{Leading order expansion in non-vacuum spacetimes}
The geometric quantities satisfy the following differential equations:
\begin{lemma}
The induced metric, the second fundamental forms and the connection 1-form satisfy the following differential equations:
\begin{equation}\label{eq_sigma}  \partial_r \sigma_{ab} = -2 l_{ab} \end{equation}
\begin{equation}\label{eq_l}\partial_r l_{ab} = R_{LabL} -l_{ac} l^c_b\end{equation}
\begin{equation}\label{eq_n} \partial_r n_{ab} = R_{Lab\underline L}  - l_b^c n_{ac} +  \nabla_a  \eta_b - \eta_a\eta_b\end{equation}
\begin{equation}\label{eq_eta}\partial_r \eta_a = R_{LaL\underline L} +l_a^b \eta_b\end{equation}
\begin{equation} \label{eq_ray}\partial_r (\sigma^{ab} l_{ab})= \frac{1}{2}(\sigma^{ab} l_{ab}) ^2 + \hat{l}_a^b \hat{l}_b^a + Ric(L,L)\end{equation}
\begin{equation}\label{eq_trn}\partial_r (\sigma^{ab} n_{ab}) =  Ric(L,\underline L)+R_{L\underline LL\underline L}+  l^{ab} n_{ab} + div_{\sigma} \eta - \eta_a \eta^a. \end{equation}
$Ric$ and $R_{\alpha \beta \gamma \delta}$ are
 the Ricci curvature and the full Riemannian curvature tensor of the spacetime N, respectively.
\end{lemma}

We have the following expansions for the curvature tensor:
\begin{equation}
\begin{split}\label{exp_R_LaLN}
 R_{LabL } =  & r^2 \bar R_{LabL} +O(r^3)  \\
R_{LaL\underline L}=& r  \bar R_{L a L\underline L} +O(r^2) \\
 R_{L\underline LL\underline L  } =&  \bar R_{L\underline LL\underline L}+O(r),
\end{split}
\end{equation}
where  $\bar R_{LabL}$, $\bar R_{LaL\underline L}$ and $\bar R_{L\underline LL\underline L  }$ correspond to the appropriate rescaled limit of the respective
tensors as $r\rightarrow 0$. For example,
\[\bar R_{LabL}=\lim_{r\rightarrow 0} \frac{1}{r^2} R_{LabL}=R(\frac{\partial}{\partial x^0}+\tilde{X}^i \frac{\partial}{\partial x^i},   \frac{\partial \tilde{X}^j}{\partial u^a},\frac{\partial \tilde{X}^k}{\partial u^b}, \frac{\partial}{\partial x^0}+\tilde{X}^l\frac{\partial}{\partial x^l})(p).\] It is considered as a $(0, 2)$ tensor on the standard $S^2$.

\begin{lemma} \label{expansion_first}
We have the following expansions:
\begin{align}\label{l} l_{ab}= & - r \tilde \sigma_{ab} +\frac{2}{3}r^3 \bar R_{La b L}+ O(r^4) \\
  \label{sigma}
\sigma_{ab} =& r^2\tilde\sigma_{ab} - \frac{1}{3} r^4\bar R_{La b L} + O(r^5)\\
\label{l_contracted}l_a^c= &-r^{-1} \delta_a^c+\frac{1}{3} r \bar R_{La b L}\tilde{\sigma}^{bc}+ O(r^2)\\
 \label{eta_first_order}\eta_a = &\frac{1}{3}r^2 \bar R_{La L\underline L}+O(r^3)\\
\sigma^{ab} l_{ab} =& - \frac{2}{r} +\frac{  1}{3} r \bar Ric(L,L)    +O(r^2) \\
\sigma^{ab} n_{ab} =& \frac{1}{r} + r [\bar R_{L\underline L L\underline L}+\frac{2}{3}\bar  Ric(L,\underline L)+\frac{1}{6} \bar Ric(L,L)]+O(r^2).
\end{align}
\end{lemma}
In summary, we have the following expansions on the surfaces $\Sigma_r$:
\begin{lemma} \label{non_physical_data}
We have the following expansions for the data $(\sigma, |H| , div \alpha_H)$ on $S^2$:
\begin{equation}
\begin{split}
\sigma_{ab} = &  r^2\tilde\sigma_{ab} - \frac{1}{3} r^4\bar R_{L a b L} + O(r^5) \\
|H|^2 = & \frac{4}{r^2} + [2 \bar R_{L\underline LL\underline L} + \frac{4}{3}\bar Ric(L,\underline L) + \frac{1}{3} \bar Ric(L,L)] + O(r)  \\
 div_\sigma   \alpha_H = &  \tilde\Delta \left [  \frac{1}{2} \bar R_{L\underline LL\underline L} + \frac{1}{6} \bar Ric(L,L)  + \frac{1}{3}\bar Ric(L,\underline L) \right ] \\
&  -\bar R_{L\underline LL\underline L} - \frac{1}{3}\bar Ric(L,\underline L)- \frac{1}{6}\bar Ric(L,L) + O(r).
\end{split}
\end{equation}
\end{lemma}
\subsection{Further expansions in vacuum spacetimes}
In this subsection, we assume the spacetime is vacuum and compute the higher order terms in the expansions for the physical data. Enough expansions are obtained to evaluate the leading term of the small sphere limit of the quasi-local energy. In a vacuum spacetime, the only non-trivial components of the curvature tensor are the Weyl curvature tensor. We have
\[  R_{\alpha \beta \gamma \delta} = W_{\alpha \beta \gamma \delta}  + \kappa^2 (g_{\alpha \gamma}g_{\beta\delta} - g_{\alpha \delta} g_{\beta \gamma}) \]
and
\[  Ric_{\alpha \delta} = -3 \kappa^2 g_{\alpha \delta}. \]
In terms of the null frame $\{ e_a, L ,  \underline L\}$, we have
\[
\begin{split}
R_{a L b L}= &W_{a  L  b L}\\
R_{a L b\underline{L}}= &W_{a\underline{L} b\underline{L}} -  \kappa^2 g_{ab}\\
R_{ a L L \underline{L}} = & W_{ a L L \underline{L}}  \\
R_{L \underline{L}  L  \underline{L}} = & W_{L \underline{L}  L  \underline{L}} - \kappa^2.
\end{split}
\]

We decompose the Weyl curvature tensor at the point $p$ using the null frame $\{ e_a, L ,  \underline L\}$ following the notation of Christodoulou and Klainerman in \cite{Christodoulou-Klainerman}:
\begin{align*}
{\alpha}_{ab} &=\bar W_{aLbL}   \  & \underline{\alpha}_{ab} =& \bar W_{a\underline{L} b\underline{L}} \\
  \beta_a & =\bar W_{a L\underline{L}L}  \    & \underline{\beta}_a=&\bar W_{a\underline{L}\underline{L} L} \\     
  \rho&=\bar W_{\underline{L}L\underline{L}L}  \ & \sigma = &\epsilon^{ab}\bar W_{ab\underline{L}L}.  
\end{align*}

From the vacuum condition and the Bianchi equations, we obtain the following relations:
\begin{equation}\label{relations}\begin{split}
\bar W_{L a b \underline L} &= \frac{1}{2} \tilde \sigma_{ab} \rho + \frac{1}{4}\epsilon_{ab}\sigma\\
\bar W_{abcL}&=-\epsilon_{ab} \epsilon_{cd} \beta^d\\
\bar W_{abc\underline L}&=\epsilon_{ab} \epsilon_{cd} \underline{\beta}^d\\
\bar W_{ab\underline{L} L}&=\frac{1}{2} \epsilon_{ab}\sigma.\end{split}\end{equation}

All $\alpha, \underline\alpha, \beta, \underline\beta, \rho$ and $\sigma$ are considered as tensors on $S^2$ through the limiting process described above. In particular, the covariant derivatives of them with respect to the standard metric $\tilde{\sigma}_{ab}$ can be computed as follows.

\begin{lemma} \cite[Lemma 3.6]{Chen-Wang-Yau2}
\begin{equation}\label{Weyl_derivatives1}
\begin{split}
\tilde{\nabla}_c \alpha_{ab}= &(\tilde{\sigma}_{ca}\tilde{\sigma}_{bd}+\tilde{\sigma}_{cb}\tilde{\sigma}_{ad}+\epsilon_{ca}
\epsilon_{bd}+\epsilon_{cb}
\epsilon_{ad})\beta^d\\
\tilde{\nabla}_c \underline{\alpha}_{ab}= &\frac{1}{2} (\tilde{\sigma}_{ca}\tilde{\sigma}_{bd}+\tilde{\sigma}_{cb}\tilde{\sigma}_{ad}+\epsilon_{ca}
\epsilon_{bd}+\epsilon_{cb}
\epsilon_{ad})\underline{\beta}^d\\
\tilde \nabla_a \beta_b = &-\frac{3}{4} \sigma \epsilon_{ab}  + \frac{3}{2} \rho \tilde \sigma_{ab}- \frac{1}{2} \alpha_{ab}\\
\tilde \nabla_a \underline{\beta}_b = &\frac{3}{8} \sigma \epsilon_{ab}  + \frac{3}{4} \rho \tilde \sigma_{ab}- \underline{\alpha}_{ab}\\
\tilde{\nabla}_a \rho= &-\beta_a-2\underline{\beta}_a\\
\tilde{\nabla}_a \sigma= &2\epsilon_{ab}(\beta^b-2\underline{\beta}^b).
\end{split}
 \end{equation}
\end{lemma}

Contracting with respect to $\tilde{\sigma}_{ab}$ and $\epsilon_{ab}$, we obtain the following formulae:
\begin{lemma} \cite[Lemma 3.7]{Chen-Wang-Yau2}
\begin{equation}\label{Weyl_derivatives2}
\begin{split}\tilde \nabla^a \alpha_{ab} = & 4 \beta_b, \epsilon^{ca}\nabla_c\alpha_{ab}=4\epsilon_{bd}\beta^d\\
\tilde \nabla^a \underline \alpha_{ab} = & 2 \underline \beta_b, \epsilon^{ca}\nabla_c\underline \alpha_{ab}=2\epsilon_{bd}\underline \beta^d\\
\tilde \nabla^a \beta_a = & 3 \rho, \epsilon^{ab}\tilde{\nabla}_a\beta_b=-\frac{3}{2}\sigma\\
\tilde \nabla^a \underline \beta_a = &\frac{3}{2} \rho, \epsilon^{ab}\tilde{\nabla}_a\underline\beta_b=\frac{3}{4}\sigma.\end{split}
\end{equation}
In particular, it follows that $\tilde{\Delta}\rho=-6\rho$ and $\tilde{\Delta}\sigma=-6\sigma$.
\end{lemma}

The covariant derivative in the spacetime $N$ at $p$ in the direction of $L$ is denoted by the symbol $D$. For example,
\[ D\alpha_{ab} = \nabla^N_L W(e_a,L,e_b,L)(p).   \]
$D\alpha_{ab}$ is also considered as a tensor on $S^2$ through the limiting process and its covariant derivatives with respect to the standard metric $\tilde{\sigma}_{ab}$ can be computed in the same manner. Relations similar to equation \eqref {Weyl_derivatives2} hold among $D$ of the Weyl curvature components.
\begin{lemma}\cite[Lemma 3.9]{Chen-Wang-Yau2}
\label{D_divergence}
\begin{equation}\begin{split}
 \tilde \nabla^a D \beta_a =  & 4 D \rho\\
 \tilde \nabla^a D^2 \beta_a =  & 5 D^2 \rho\\
\tilde\nabla^a(D\alpha_{ab})= & 5D\beta_b\\
\tilde\nabla^a(D^2\alpha_{ab})=& 6 D^2\beta_b\end{split}
\end{equation}
\end{lemma}
We have the following expansions for the Weyl curvature tensor.  
\begin{lemma}\cite[Lemma 3.10]{Chen-Wang-Yau2}
\begin{equation}\begin{split}\label{exp_LaLN} W_{LaL\underline L}&= r  \beta_a +r^2 D\beta_a +\frac{1}{2} r^3 D^2\beta_a+O(r^4) \\
 W_{L\underline LL\underline L  } &=\rho + r D \rho+r^2 [ \frac{1}{2} D^2 \rho -\frac{1}{3} |\beta|^2]+O(r^3). \end{split}
\end{equation}
\end{lemma}
We are now ready to compute the expansion of the physical data:
\begin{lemma} \label{data}
We have the following expansions for  $\sigma^{ab} l_{ab}$, $\sigma^{ab} n_{ab}$ and $\eta_a$.
\begin{equation}
\sigma^{ab} l_{ab} = - \frac{2}{r} +\frac{  1}{45} r^3 |\alpha|^2 +O(r^4)
\end{equation}
\begin{equation}
\sigma^{ab} n_{ab} = \frac{1}{r} + r(\sigma^{ab} n_{ab})^{(1)}+r^2 (\sigma^{ab} n_{ab})^{(2)}+r^3(\sigma^{ab} n_{ab})^{(3)}+O(r^4)
\end{equation}
and
\begin{equation}\label{exp_eta} \eta_a = \frac{r^2}{3} \beta_a +\frac{r^3}{4}   D\beta_a +r^4 [\frac{1}{10}   D^2 \beta_a  -\frac{1}{45} \alpha_{ab} \beta^b ] +O(r^5), \end{equation}
where
\begin{equation}
\begin{split}
(\sigma^{ab} n_{ab})^{(1)}=&\rho +  \kappa^2 \\
(\sigma^{ab} n_{ab})^{(2)}=&\frac{2}{3}D\rho\\
(\sigma^{ab} n_{ab})^{(3)}=&   \frac{3}{8}D^2\rho+\frac{1}{30}|\alpha|^2- \frac{11}{45}|\beta|^2 .
\end{split}
\end{equation}
\end{lemma}
\begin{proof}

We  rewrite $l_{ab}$ as
\begin{align*}
l_{ab} =& -r \tilde \sigma_{ab} -\frac{2}{3}r^3\alpha_{ab}+ O(r^4)\\
=& (-r \tilde \sigma_{ab} -\frac{1}{3}r^3\alpha_{ab}) - \frac{1}{3}r^3\alpha_{ab}+ O(r^4).
 \end{align*}
Hence, $\hat l_{ab}$, the traceless part of $l_{ab}$, is given by
\begin{equation}
\hat l_{ab} =-\frac{1}{3}r^3\alpha_{ab} + O(r^4).
\end{equation}
It follows that
\begin{equation}  \sigma^{ab} l_{ab} = - \frac{2}{r} +\frac{  r^3}{45} |\alpha|^2    +O(r^4). \end{equation}

Next we compute $\eta_a$. Let
\[ \eta_a = r^2 \eta_a^{(2)} +r^3 \eta_a^{(3)} +r^4 \eta_a^{(4)} +O(r^5). \]
From Lemma \ref{expansion_first}, we have
\[  \eta_a = \frac{1}{3}r^2 \beta_a+O(r^3). \]
Equation \eqref{eq_eta} is equivalent to
\[r^{-1}\partial_r(r\eta_a)=W_{LaL\underline L}+(l_a^b+r^{-1}\delta_a^b)\eta_b.\]
By equation \eqref{exp_LaLN}, the right hand side can be expanded into
\[r \beta_a+r^2 D\beta_a+r^3[\frac{1}{2}D^2\beta_a-\frac{1}{9}\alpha_{ab}\beta^b]+O(r^4).\]
Integrating, we obtain
\begin{align*}
 \eta_a^{(3)}  = &\frac{1}{4}   D\beta_a  \\
  \eta_a^{(4)} = &\frac{1}{10}   D^2 \beta_a -\frac{1}{45}\alpha_{ab}\beta^b.
\end{align*}

For $n_{ab}$, we start with equation \eqref{eq_n}. It is equivalent to
\[
\begin{split}
r\partial_r (r^{-1} n_{ab}) = & R_{Lab\underline L}  - (l_b^c+r^{-1}\delta_b^c) n_{ac} +  \nabla_a  \eta_b - \eta_a\eta_b  \\
=  & W_{Lab\underline L} + \sigma_{ab} - (l_b^c+r^{-1}\delta_b^c) n_{ac} +  \nabla_a  \eta_b - \eta_a\eta_b.
\end{split}
\]
The equation becomes
\[
\begin{split}
  r\partial_r (r^{-1} n_{ab}) =&r^2[\bar W_{Lab\underline L} +  \kappa^2 \tilde \sigma_{ab}  -\frac{1}{6} \bar W_{LabL} +   \tilde \nabla _a  \eta^{(2)}_b]+O(r^3)\\
=&r^2 (\rho + \kappa^2)\tilde \sigma_{ab}+O(r^3).
\end{split}
\]
Integrating, we obtain
\[ n_{ab} =\frac{1}{2} r \tilde{\sigma}_{ab} + \frac{1}{2} r^3  (\rho + \kappa^2) \tilde \sigma_{ab} + O(r^4).  \]

Lastly, we deal with equation \eqref{eq_trn}  for $\sigma^{ab}n_{ab}$. We decompose
\begin{equation}
\begin{split}
l^{ab} n_{ab} = & l_{ab} \sigma^{ac} \sigma^{bd} n_{cd} \\
                     =& (l_{ab}+\frac{\sigma_{ab}}{r}) \sigma^{ac} \sigma^{bd} (n_{cd} - \frac{\sigma_{cd}}{2r}) + \frac{1}{2} r^{-1}\sigma^{ab} l_{ab} - r^{-1}\sigma^{ab}n_{ab} +  r^{-2} .       \\
                     \end{split}
\end{equation}
Thus equation \eqref{eq_trn} is equivalent to
\[r^{-1}\partial_r (r\sigma^{ab} n_{ab}) = r^{-2}+\frac{1}{2}r^{-1}\sigma^{ab} l_{ab}+  (l^{ab}+r^{-1}\sigma^{ab}) (n_{ab}-\frac{1}{2} r^{-1}\sigma_{ab})- \eta_a\eta^a + W_{L\underline LL\underline L} +2 \kappa^2+ div_{\sigma} \eta. \]
Notice that \[
\begin{split}
l_{ab}+r^{-a}\sigma_{ab}=&-\frac{1}{3}r^3\alpha_{ab}+O(r^4)\\
n_{cd}-\frac{1}{2} r^{-1}\sigma_{cd}
= &r^3( \frac{1}{2}(\rho + \kappa^2)\tilde \sigma_{ab}- \frac{1}{6}\alpha_{ab})+O(r^4).
\end{split}
\]
We have
\[\begin{split}
r^{-1}\partial_r (r \sigma^{ab} n_{ab})
= & r^ 2 \left [\frac{1}{15}|\alpha|^2-  \frac{1}{9} |\beta|^2   \right ]   + div_{\sigma} \eta+ W_{L\underline LL\underline L}+2 \kappa^2 +O(r^3).
\end{split}
\]Integrating this equation, we obtain the expansion for $\sigma^{ab}n_{ab}$.

\[
\begin{split}
\sigma^{ab} n_{ab} =& \frac{1}{r} + \frac{r}{2}[(\rho +2 \kappa^2+ (div_{\sigma} \eta)^{(0)}] + \frac{r^2}{3}[ D \rho+ (div_{\sigma} \eta)^{(1)}]+ \frac{r^3}{4} \Big [ \frac{1}{2} (D^2 \rho -\frac{2}{3} |\beta|^2) \\
&+ (div_{\sigma} \eta)^{(2)}-\frac{1}{9}|\beta|^2+ \frac{1}{15}  |\alpha|^2  \Big ] + O(r^4).
\end{split}
\]

We compute
\begin{equation}\label{divergence_null_connection}
\begin{split}
(div_{\sigma} \eta)^{(0)}=&\frac{1}{3}\tilde{\nabla}^a(
\beta_a)=\rho \\
(div_{\sigma} \eta)^{(1)}=&\frac{1}{4} \tilde{\nabla}^a D \beta_a = D \rho \\
(div_{\sigma} \eta)^{(2)}=&\frac{1}{10} \tilde{\nabla}^a  (D^2 \beta_a ) -\frac{2}{15}\tilde{\nabla}^a ( \alpha_{ab} \beta^b)\\
=&\frac{1}{2}  D^2\rho+\frac{1}{15}( |\alpha|^2- 8 |\beta|^2).
\end{split}
\end{equation}
which follows from the expansion of $\eta$ and the following expansion for $\gamma_{ab}^c$:
\begin{equation}\label{gamma} \gamma_{ab}^c = \tilde \gamma_{ab}^c +r^2 \gamma_{ab}^{(2)c}+O(r^3) \end{equation}
where $ \tilde \gamma_{ab}^c $ is the  Christoffel symbols for $\tilde \sigma_{ab}$ and $\gamma_{ab}^{(2)c}= -\frac{1}{6} \tilde\sigma^{cd}(\tilde \nabla_a \alpha_{db}+ \tilde \nabla_b \alpha_{ad} - \tilde \nabla_d \alpha_{ab})$.
\end{proof}
When we compute the small sphere limit in vacuum spacetimes, there are several functions and tensors on $S^2$ which appear repeatedly. These quantities are computed in \cite{Chen-Wang-Yau2}. We recall the results here. We define the functions $W_0$, $W_i$ and $P_k$ as follows:
\begin{equation}\label{W_P}
\begin{split}
W_0=&  \tilde X^i \tilde X^j \bar W_{0i0j} =\rho\\
W_i=  & \tilde X^j \tilde X^k \bar W_{0kij}=\frac{1}{2}(\beta^b-2\underline{\beta}^b)\tilde{\nabla}_b \tilde{X}^i \\
P_k= & \frac{1}{15}\bar W_{0i0k} \tilde X^i-\frac{1}{6}W_0\tilde X^k=-\frac{1}{30}(\beta^a+2\underline{\beta}^a)\tilde{\nabla}_a \tilde{X}^k-\frac{1}{10} \rho\tilde{X}^k\\
\end{split}
\end{equation}
$W_i$ are $-6$-eigenfunctions and $P_k$ are $-12$-eigenfunctions of the standard Laplacian on $S^2$.

Next, we introduce $R_{ij}$ and $S_j$. From the expansion of the induced metric $\sigma_{ab}$, we derive
\begin{equation}
\sigma^{(0) ab}=-\frac{1}{3}\alpha^{ab} \text{ and  }
\tilde{\sigma}^{ab} \gamma_{ab}^{(2)c} =-\frac{4}{3}  \beta^c.
\end{equation}

$R_{ij}$ and $S_j$ are defined as follows:
\begin{equation}\label{R_S}
\begin{split}
R_{ij}=&\sigma^{(0) ab}\tilde{X}^i_a \tilde{X}^j_b \\
= &  \frac{1}{3}[2\tilde X^i\tilde X^k \bar W_{0i0k}
 + 2\tilde {X}^j\tilde {X}^k \bar W_{0k0i} +  \tilde X^i\tilde X^j\tilde X^n(\bar W_{0inj}+\bar W_{0jni}) \\
& -2\bar W_{0i0j}- \rho \delta_{ij} -\rho \tilde X^i\tilde X^j-\tilde X^n(\bar W_{0inj}+\bar W_{0jni}) ]\\
S_j=&\tilde{\sigma}^{ab} \gamma_{ab}^{(2)c} \tilde{X}^j_c\\
= &\frac{1}{3}(-4   \bar W_{0j0n}\tilde X^n + 4 \tilde X^j W_0 + 4 W_j). 
\end{split}
\end{equation}


\section{Optimal embedding equation}
For the rest of the paper, we set $\kappa=1$. This corresponds to choosing $\Lambda=-3$ in the Einstein equation. In this section, we study the limiting behavior of the optimal embedding equation.
We consider the Ads space to be embedded in $\R^{3,2}$ and consider isometric embeddings of the form:
\[ Y= (Y_0,Y_i, Y_4) \]
where $Y_0$, $Y_i$ and $Y_4$ are the restriction of $y_0$ and $y_i$ and $y_4$ to the image of the isometric embedding.

Similar to the small sphere limit of the Wang--Yau quasi-local mass, we look for solutions of the optimal embedding equation $(Y,T_0)$ of the form
\begin{equation}\label{assume}
\begin{split}
Y_0  = & \sum_{i=3}^{\infty}Y_0^{(i)}r^i \\
Y_i = &  r \tilde X^i+ \sum_{k=3}^{\infty} Y_i^{(k)}r^k  \\
Y_4 =  & \sqrt{1 + \sum_i Y_i^2 - Y_0^2}
\end{split}
\end{equation}
and
\[
T_0=  A (y^0\frac{\partial}{\partial y^4}-y^4\frac{\partial}{\partial y^0} )- B_i( y^0\frac{\partial}{\partial y^i}+y^i\frac{\partial}{\partial y^0})- C_j (y^4\frac{\partial}{\partial y^j}+y^j\frac{\partial}{\partial y^4})+ D_p \epsilon_{pqr}y^q\frac{\partial}{\partial y^r}.
\]
\begin{remark}
As mentioned in Remark \ref{remark_2}, the quasi-local energy is equivariant and thus it is natural to consider the embedding of the above form and allow $T_0$ to be a general observer Killing field.
\end{remark}
In this section, we compute the leading order term of the optimal embedding equation and use it to determine $Y_i^{(3)}$ . While it is possible to also solve $Y_0^{(3)}$, the set of solutions is different for the vacuum case and the non-vacuum case. Hence, we defer the determination of $Y_0^{(3)}$ to  Section 8 for vacuum spacetimes since it will only be needed for the limit in the vacuum case.

Recall that the isometric embedding equation of a metric $\sigma$ into the AdS space in $\R^{3,2}$ is
\[
\sum_i \partial_a Y_i \partial_b Y_i - \partial_a Y_0\partial_b Y_0 - \partial_a Y_4\partial_b Y_4 = \sigma_{ab}.
\]
For the metric $\sigma_{ab}$ on $\Sigma_r$ given in \eqref{sigma} and the embedding $(Y_0, Y_i,Y_4)$ given in \eqref{assume}, the leading order term of the isometric embedding is the following system of linear equation on $Y_i^{(3)}$
\begin{equation} \label{linearized_iso_equ}
\partial_a \tilde X^i \partial_b Y_i^{(3)}  + \partial_b \tilde X^i  \partial_a Y_i^{(3)}   = -\frac{1}{3}\bar R_{LabL} =  \frac{1}{3} \alpha_{ab} - \frac{1}{6} \bar Ric(L,L) \tilde \sigma_{ab}
\end{equation}
\begin{lemma}
\label{yi3}
\[ \begin{split}Y^{(3)}_i=-\frac{1}{3}\beta^c \tilde{\nabla}_c \tilde{X}^i+\frac{1}{2} \rho\tilde{X}^i -\frac{1}{12} \bar Ric(L,L) \tilde X^i
 \end{split} \]
satisfies \eqref{linearized_iso_equ}.
\end{lemma}
\begin{proof}
From Lemma 6.1 of \cite{Chen-Wang-Yau3}, we know that
\[ \sum_{i}\partial_a \tilde X^i \partial_b (-\frac{1}{3}\beta^c \tilde{\nabla}_c \tilde{X}^i+\frac{1}{2} \rho\tilde{X}^i ) + \partial_b \tilde X^i  \partial_a (-\frac{1}{3}\beta^c \tilde{\nabla}_c \tilde{X}^i+\frac{1}{2} \rho\tilde{X}^i )   =  \frac{1}{3} \alpha_{ab}.
 \]
On the other hand, it is easy to see that
\[  
\sum_{i} \partial_a \tilde X^i \partial_b ( \bar Ric(L,L) \tilde X^i) + \partial_b \tilde X^i  \partial_a ( \bar Ric(L,L) \tilde X^i) = 2 \bar Ric(L,L) \tilde \sigma_{ab}.
 \]
 This finishes the proof of the lemma.
\end{proof}
The data on $\Sigma_r$ admit the following expansion
\[
\begin{split}
\sigma_{ab} =& r^2 \tilde \sigma_{ab} + r^4\tilde \sigma^{(4)}_{ab} + O(r^5)\\
|H| = & \frac{2}{r} + h^{(1)} r + O(r^2)\\
(\alpha_H )_a= & (\alpha_H^{(2)})_a r^2 + O(r^3).
\end{split}
\]
Similarly. the data on the image of the isometric embedding admit the following expansion
\[
\begin{split}
|H_0| = & \frac{2}{r} + h_0^{(1)} r + O(r^2)\\
(\alpha_{H_0} )_a= & (\alpha^{(2)}_{H_0})_a r^2 + O(r^3).
\end{split}
\]
It follows that the quasi-local energy density $f$ admits the following expansion:
\[
f = f^{(1)} r + O(r^2)
\]
where $f^{(1)}= \frac{h_0^{(1)}-  h^{(1)}}{A}$. In the following lemma, we derive the leading order term of the Euler--Lagrange equation for the quasi-local energy.
\begin{lemma} \label{optimal_eq_leading}
For the observer \[
T_0=  A (y^0\frac{\partial}{\partial y^4}-y^4\frac{\partial}{\partial y^0} )- B_i( y^0\frac{\partial}{\partial y^i}+y^i\frac{\partial}{\partial y^0})- C_j (y^4\frac{\partial}{\partial y^j}+y^j\frac{\partial}{\partial y^4})+ D_p \epsilon_{pqr}y^q\frac{\partial}{\partial y^r}
\]
and an isometric embedding of the form \eqref{assume}, the leading order term of the Euler--Lagrange equation is the following equation on $(A,C^i)$ and  $Y_0^{(3)}$
\[
    \frac{1}{2}  \tilde \Delta  (  \tilde \Delta + 2 )  Y_0^{(3)} =  \tilde \nabla^a (\alpha_H^{(2)})_a  + \tilde \nabla^a ( f^{(1)} \tilde \nabla_a C_i\tilde X^i )+ \frac{1}{2}  \tilde \Delta  (f^{(1)} C_i\tilde X^i ).
\]
\end{lemma}
\begin{proof}
From  \cite[Theorem 5.4 ]{Chen-Wang-Yau3}, the first variation of the quasi-local energy (up to a factor of $8 \pi$) is
\begin{equation}\label{first_variation}
\begin{split}
&  \int_{\Sigma} (\delta \tau)   div\left [  V ^2 \nabla \sinh^{-1} \frac{f div (V ^2 \nabla \tau)}{|H_0||H|}  - f V ^4 \nabla \tau +V ^2(\alpha_{H_0} - \alpha_H) \right ]   d \Sigma \\
 & + \int_{\Sigma} \delta Y^i \bar\nabla_i V \left [ f  V (1+ 2 V ^2|\nabla \tau|^2)  -2  V \nabla \tau \nabla \sinh^{-1} \frac{f div (V ^2 \nabla \tau)}{|H_0||H|}  + (\alpha_{H} - \alpha_{H_0})(2 V \nabla \tau)  \right ]   d \Sigma
\end{split}
\end{equation}
where $\delta \tau$ and $\delta Y^i $ are coupled through the isometric embedding equation. Recall that
\[
\begin{split}
V ^2 = & - \langle T_0, T_0 \rangle  = (A^2 - \sum_i C_i^2) + O(r^2)\\
 V ^2 \nabla_a \tau = &-(T_0^T)_a = r C_i \tilde \nabla_a \tilde X^i + O(r^2).
 \end{split}
\]
As a result,
\[
div\left [  V ^2 \nabla \sinh^{-1} \frac{f div (V ^2 \nabla \tau)}{|H_0||H|}  - f V ^4 \nabla \tau +V ^2(\alpha_{H_0} - \alpha_H) \right ]  = O(1)
\]
while
\[
 f  V (1+ 2 V ^2|\nabla \tau|^2)  -2  V \nabla \tau \nabla \sinh^{-1} \frac{f div (V ^2 \nabla \tau)}{|H_0||H|}  + (\alpha_{H} - \alpha_{H_0})(2 V \nabla \tau)   = O(r).
\]
Moreover, from the linearized isometric embedding equation, we conclude that
\[  \delta Y^i = O(r^2)  . \]
As a result, the second integral in \eqref{first_variation} is of higher order in $r$ compared to the first integral and the leading order term of the Euler--Lagrange equation is simply that the O(1) term of
\[
div\left [  V ^2 \nabla \sinh^{-1} \frac{f div (V ^2 \nabla \tau)}{|H_0||H|}  - f V ^4 \nabla \tau +V ^2(\alpha_{H_0} - \alpha_H) \right ]
\]
is equal to 0. From the formula in the proof of Theorem 6.2 of \cite{Chen-Wang-Yau3}, we get  
\begin{equation} \label{div_alpha} div \alpha_{H_0}  =  \frac{1}{2}  \tilde \Delta  (  \tilde \Delta + 2 )  Y_0^{(3)}\end{equation}
by treating the image of the isometric embedding as a small perturbation of the isometric embedding into the hyperbolic space. The lemma follows from collecting terms.
\end{proof}
\section{Small sphere limit of the quasi-local energy in spacetimes with matters}
In this section, we compute the small sphere limit of the quasi-local energy and show that it recovers the matter field at $p$. More precisely, we show the following:
\begin{theorem} \label{thm_small_non}
Let $\Sigma_r$ be the family of spheres approaching $p$ constructed in Section 4 and $\mathcal P$ denote the set of $(Y,T_0)$ admitting a power series expansion given in equation \eqref{assume}.
\begin{enumerate}
\item 
For any pair $(Y,T_0)$ in $\mathcal P$, we have
\[ \lim_{r \to 0} r^{-3}E(\Sigma_r, Y_r,T_0) = \frac{4\pi}{3}T(e_0, A e_0 + C_i e_i), \] where $T(\cdot, \cdot)$ is the stress-energy  tensor at $p$
\item Suppose $T(e_0, \cdot) $ is dual to a future directed timelike vector $W$ at $p$.  We have 
\[ \inf_{(Y,T_0) \in \mathcal P} \lim_{r \to 0} r^{-3} E(\Sigma_r, Y_r,T_0) = \frac{4\pi}{3} \sqrt{-\langle W, W\rangle}. \] 
The infimum is achieved by a unique $(A,C_i)$.
\end{enumerate}
\end{theorem}
\begin{proof}
The quasi-local energy is 
\[
 \frac{1}{8\pi}\int_\Sigma \left[f(V ^2+V ^4|\nabla\tau|^2)+ div(V ^2 \nabla \tau) \sinh^{-1}(\frac{f div(V ^2 \nabla \tau) }{|H_0| |H|})-\alpha_{H_0}(V ^2 \nabla \tau)+\alpha_H(V ^2 \nabla \tau)\right] d\Sigma 
\]
where
\[
\begin{split}
f =&  \frac{h_0^{(1)} -  h^{(1)} }{A}r + O(r^2)\\
V ^2 =& A^2 - \sum_i C_i^2   + O(r)\\
 V ^2 \nabla_a \tau = & r C_i \tilde \nabla_a \tilde X^i + O(r^2).
\end{split}
\]
It is easy to see that
\[
E(\Sigma, Y, T_0) = (A e + C_i p^i) r^3 + O(r^4)
\]
where
\[ e = \frac{1}{8 \pi}  \int_{\Sigma_r}(h_0^{(1)}-h^{(1)}) dS^2 \]
\[ p^i = \frac{1}{8 \pi}  \int_{\Sigma_r} \tilde X^i  \tilde \nabla^a (\alpha_H^{(2)} -\alpha_{H_0}^{(2)})_a   dS^2.\]
It suffices to show that
\begin{align}
\label{limit_matter_energy} \frac{1}{8 \pi}  \int_{\Sigma_r}(h_0^{(1)}-h^{(1)}) dS^2 =& \frac{4\pi}{3}T(e_0, e_0)\\
\label{limit_matter_momentum} \frac{1}{8 \pi}  \int_{\Sigma_r} \tilde X^i  \tilde \nabla^a (\alpha_H^{(2)} -\alpha_{H_0}^{(2)})_a   dS^2 = &  \frac{4\pi}{3}T(e_0,  e_i)
\end{align}
For \eqref{limit_matter_energy}, we consider the Gauss curvature $K$ of $\Sigma_r$. It admits the following expansion
\[
2 \sqrt{K} = \frac{2}{r} + k^{(1)} r +O(r^2)
\]
From Lemma 4.1 of \cite{Chen-Wang-Yau3}, we have
\[
 \frac{1}{8 \pi}  \int_{\Sigma_r}(k^{(1)}-h^{(1)}) dS^2 = \frac{4 \pi}{3} (\bar Ric (e_0,e_0) + \frac{1}{2}\bar R )
 \]
On the other hand, from the Gauss equation of image of the isometric embedding of $\Sigma_r$ in the hyperbolic space, we have
\[
K= -1 +\frac{1}{4} |H_0|^2 + O(r^2)
\]
and thus
\[
k^{(1)} = h_0^{(1)}- 1.
\]
As a result,
\[   \int_{\Sigma_r}(h_0^{(1)}-h^{(1)}) dS^2  =  \frac{4 \pi}{3} (\bar Ric (e_0,e_0) + \frac{1}{2}\bar R  + 3) = 8 \pi T(e_0,e_0). \]
Next we compute $p^i$.  From \eqref{div_alpha}, we conclude that
\[   \int_{S^2} \tilde X^i   \tilde \nabla^a (\alpha_{H_0}^{(2)})_a   dS^2  =0.\]
Using the expansion for $div_\sigma \alpha_H$ from Lemma \ref{non_physical_data}, we have
\begin{align*}
  p^i =  r^3 \int_{S^2} &   {\Big [}  \tilde\Delta[  \frac{1}{2} \bar R_{L\underline LL\underline L} + \frac{1}{6} \bar Ric(L,L)  + \frac{1}{3}\bar Ric(L,\underline L)] \\
&  -\bar R_{L\underline LL\underline L} - \frac{1}{3}\bar Ric(L,\underline L)- \frac{1}{6}\bar Ric(L,L){\Big ]} \tilde X^idS^2 + O(r^{4}) \\
 =r^3 \int_{S^2} &   {\Big [}  -2\bar R_{L\underline LL\underline L} - \bar Ric(L,\underline L)- \frac{1}{2}\bar Ric(L,L){\Big ]} \tilde X^idS^2 + O(r^{4})
\end{align*}
where we apply integration by parts for the last equality.

To evaluate the above integral, we switch to the orthogonal frame $\{e_0 , e_i\}$. We have
\begin{align*}
  \int_{S^2}  {\Big [}  -2\bar R_{L\underline LL\underline L} - \bar Ric(L,\underline L)- \frac{1}{2}\bar Ric(L,L){\Big ]} \tilde X^idS^2
=&  - \int_{S^2}   \bar Ric(e_0,e_j)\tilde {X}^j \tilde {X}^idS^2 \\
= & -\frac{4 \pi}{3}\bar Ric(e_0,e_i).
\end{align*}
This finishes the proof of part (1) since $ \bar Ric(e_0,e_i) = 8 \pi T_{0i}.$

For part (2), we recall that from \eqref{obser_inequlaity}, we have
\[  A \ge \sqrt{1+ |\vec{C}|^2} \ge 1. \]
Suppose $T(e_0, \cdot) $ is dual to a future directed timelike vector $V$. Fixing $\vec{C}$, $T(e_0, A e_0 + C_i e_i) $  is strictly increasing in $A$ and the minimum of the quasi-local energy can only occur on the set of observers, $O$, such that  
\[  A = \sqrt{1+ |\vec{C}|^2}. \]
From here the Theorem easily follows.
\end{proof}

\begin{remark}
While the proof is similar to  \cite[Theorem 4.1]{Chen-Wang-Yau2} of the small sphere limit of the Wang-Yau quasi-local energy, The proof there used \cite[Theorem 2.1]{Wang-Yau3} which evaluates the limit of the quasi-local energy under the compatibility condition of mean curvature
\[
\lim_{r \to 0} \frac{|H|}{|H_0|}=1.
\]
However, we do not have the corresponding result when the AdS space is used as the reference. It will be interesting to establish the analog of \cite[Theorem 2.1]{Wang-Yau3} for the newly defined quasi-local energy.
\end{remark}

\section{Small sphere limit of the quasi-local energy in vacuum spacetimes}
In this section, we begin the computation of the small sphere limit of the quasi-local energy in vacuum spacetimes.
First, we derive the following lemma about the Gauss curvature of $\Sigma_r$.
\begin{lemma}
For vacuum spacetimes, the Gauss curvature $K$ of $\Sigma_r$ admits the following expansion
\[ K = \frac{1}{r^2} + 2 \rho  + O(r) \]
\end{lemma}
\begin{proof}
Recall that the induced metric $\sigma$ is
\[  \sigma_{ab} = r^2 \tilde \sigma_{ab} + \frac{r^4}{3} \alpha_{ab} + O(r^5)  \]
and $\alpha_{ab}$ is traceless. As a result, we have
\[
K = \frac{1}{r^2} + \frac{1}{6} \tilde \nabla^a \tilde \nabla^b \alpha_{ab} + O(r).
\]
The lemma now follows from Lemma 5.6.
\end{proof}
Next, we solve the optimal embedding equation derived in Lemma \ref{optimal_eq_leading}.
\begin{lemma}\label{y03}
For the observer $T_0= (A,\vec{B},\vec{C},\vec{D})$, the solution of the optimal embedding equation gives
\[\begin{split} Y_0^{(3)}= - \frac{1}{3}W_0+\frac{\sum_i C_iP_i}{A}.\end{split}\]
\end{lemma}
\begin{proof}
Recall that for the vacuum spacetime, we have
\[
|H| = \frac{2}{r} + (W_0 +1) r + O(r^2).
\]
On the other hand, for $|H_0|$, it suffices to compute the mean curvature
of the isometric embedding into the hyperbolic space. From the Gauss equation of the surface into the hyperbolic space, we conclude that
\[
K =  -1 + \frac{1}{4} |H_0|^2  + O(r^2)
\]
As a result, from the above lemma, we have
\[
|H_0|=  \frac{2}{r} + (2W_0 +1) r + O(r^2)
\]
Hence,
\[ f^{(1)} = \frac{W_0}{A}.\]
The rest of the proof is the same as Lemma 6.3 of \cite{Chen-Wang-Yau3}.
\end{proof}
For each $T_0$, we shall compute $8 \pi E(\Sigma_r,Y_r(T_0),T_0)$ which is given by
\begin{equation}
\begin{split}
 \label{energy_expression}
  &\int_{\Sigma_r } f(V ^2 + V ^4|\nabla \tau|^2) +( div V ^2 \nabla \tau)  \sinh^{-1} ( \frac{f  div V ^2 \nabla \tau}{|H||H_0|})  d\Sigma_r \\
  &- \int_{\Sigma_r}  \alpha_{H_0}(V ^2 \nabla \tau) d\Sigma _r + \int_{\Sigma_r}  \alpha_{H}(V ^2 \nabla \tau)d\Sigma_r.
 \end{split}
\end{equation}
We evaluate the three integrals in the next three sections, respectively and put the results together in Section \ref{sec_eva_energy}.
\section{The energy component} \label{sec_energy_com_va}
In this section, we evaluate the first integral in \eqref{energy_expression}:
\[ \int_{\Sigma_r}  f (V ^2 + V ^4|\nabla \tau|^2) + div (V ^2 \nabla \tau)  \sinh^{-1} ( \frac{f div (V ^2 \nabla \tau)}{|H||H_0|}) d\Sigma_r.\]
It suffices to evaluate $ \int_{\Sigma_r}  f [V ^2 + V ^4 |\nabla \tau|^2 +  \frac{( div (V ^2 \nabla \tau))^2}{|H||H_0|}]  d \Sigma_r$ since for $x$ small,
\[  \sinh^{-1}(x) =x + O(x^3). \]

Denote the expansion of the physical data by
\[
\begin{split}
\sigma_{ab} = & r^2\tilde \sigma_{ab} +r^4 \sigma_{ab}^{(4)}+r^5 \sigma_{ab}^{(5)}+O(r^6)\\
|H| = & \frac{2}{r} + r h^{(1)} + r^2 h^{(2)}+ r^3 h^{(3)}+O(r^4) \\
\alpha_H = & r^2 \alpha_H^{(2)}+ r^3 \alpha_H^{(3)}+ r^4 \alpha_H^{(4)}+ O(r^{5}).
\end{split}
\]
Similarly, we have
\[
\begin{split}
|H_0| = & \frac{2}{r} + r h_0^{(1)} + r^2 h_0^{(2)}+ r^3 h_0^{(3)}+O(r^4) \\
\alpha_{H_0} = & r^2 \alpha_{H_0}^{(2)}+ r^3 \alpha_{H_0}^{(3)}+ r^4 \alpha_{H_0}^{(4)}+ O(r^{5}).
\end{split}
\]

First we derive the following lemma.
\begin{lemma}\label{lemma_7_1}
\[
\begin{split}
V ^2 = &  (A^2 - \sum_i C_i^2) + 2(AB_i \tilde X^i - \sum_i C_i D_p \epsilon_{pqi} \tilde X^q) r \\
                     & + r^2 [ (B_i \tilde X^i)^2 +(C_i \tilde X^i)^2  - \sum_i (D_p \epsilon_{pqi} \tilde X^q)^2+  (A^2 - \sum_i C_i^2)] +O(r^3) \\
V ^4 |\nabla \tau| ^2   = & C_iC_j (\delta^{ij}  -\tilde {X}^i \tilde {X}^j) + 2r  \sum_i  C_iD_p \epsilon_{pqi} \tilde{X}^q  \\
   &+( (\delta^{ij}  -\tilde {X}^i \tilde {X}^j)C_iC_j+ \sum_i (D_p \epsilon_{pqi} \tilde X^q)^2+g_1) r^2 + O(r^{3})\\
   (div ( V ^2 \nabla \tau))^2    =& 4C_i C_j(\tilde {X}^i \tilde {X}^j) r^{-2} +  (4 C_iC_j \tilde {X}^i \tilde {X}^j +g_2) + O(r),
\end{split}
\]
where
\[\begin{split}
g_1& = \sum_{i,j} C_i C_j (R_{ij}+2 \tilde \nabla \tilde {X}^i \tilde \nabla  Y_j^{(3)}  )+2 A C_i \tilde \nabla\tilde  X^i \tilde \nabla Y_0^{(3)} \\
g_2& = \sum_{i,j} 4C_i C_j \tilde {X}^i( S_j-\tilde{\Delta} X_j^{(3)} )-4AC_i \tilde{X}^i\tilde{\Delta} Y_0^{(3)},
\end{split}\] and $R_{ij}$ and $S_j$ are defined in \eqref{R_S}.
\end{lemma}
\begin{proof}
We have
\[
V ^2 = - \langle T_0, T_0 \rangle = (A Y^0 + C_i Y^i) ^2 + (AY^4 + B_i Y^i)^2 -\sum_i (C_i Y^4 + B_iY^0 + D_p \epsilon_{pqi} Y^q)^2
\]
where $Y^0 = O(r^3)$, $Y^i = r \tilde{X}^i + O(r^3)$ and $Y^4 = \sqrt{1+r^2} + O(r^3)$. This gives the expansion for $V ^2$.

For the other two terms, we have
\[  V ^2 \nabla \tau= -T_0^\perp  = (A Y^0 + C_i Y^i)\nabla Y^4 + (AY^4 + B_i Y^i)\nabla Y^0 + (C_i Y^4 + B_i Y^0 + D_p  \epsilon_{pqi}  Y^q)\nabla Y^i. \]
As a result,
\[
\begin{split}
 V ^4  |\nabla \tau| ^2   = &   (C_i Y^4 + B_i Y^0 + D_p  \epsilon_{pqi}  Y^q) (C_j Y^4 + B_j Y^0 + D_m  \epsilon_{mnj}  Y^n) (\delta^{ij}  -\tilde {X}^i \tilde {X}^j)  \\
  &+  r^2 \left  [ \sigma^{(0) ab}(C_i \tilde{X}_a^i)(C_j \tilde{X}_b^j)+2(C_i \tilde \nabla \tilde  Y^i) (C_j \tilde \nabla Y_j^{(3)}+A\tilde \nabla Y_0^{(3)}) \right  ]+ O(r^{3}), \end{split}\] and the formula follows from \eqref{R_S} and \eqref{assume}.
Similarly,
\[
\begin{split}
 (div V ^2 \nabla \tau)^2    = &4\sum_{ij}(C_i Y^4 + B_i Y^0 + D_p  \epsilon_{pqi}  Y^q) (C_j Y^4 + B_j Y^0 + D_m  \epsilon_{mnj}  Y^n) (\tilde{X}^i \tilde{X}^j) r^{-2}  \\
 &+  4(C_i \tilde{X}^i)(C_j \tilde{\sigma}^{ab}\gamma_{ab}^{(2)c}\tilde{X}_c^j)-4(C_i \tilde{X}^i) ( C_j  \tilde \Delta Y_j^{(3)}+A  \tilde \Delta Y_0^{(3)})  + O(r),
 \end{split}
 \]
where
\[\begin{split} & 4(C_i \tilde{X}^i)(C_j \tilde{\sigma}^{ab}\gamma_{ab}^{(2)c}\tilde{X}_c^j)-4(C_i \tilde{X}^i) (C_j \tilde{\Delta} Y_j^{(3)}+A\tilde{\Delta} Y_0^{(3)})\\
= &4C_iC_j \tilde{X}^i S_j-4(C_i \tilde{X}^i) ( C_j \tilde{\Delta} Y_j^{(3)}+A\tilde{\Delta} Y_0^{(3)})\\
= & 4 C_iC_j \tilde{X}^i( S_j-  \tilde{\Delta} Y_j^{(3)})-4AC_i \tilde{X}^i  \tilde{\Delta} Y_0^{(3)}.\end{split}\]
\end{proof}
With the above lemma, we compute $ f [V ^2 + V ^4 |\nabla \tau|^2 +  \frac{( div (V ^2 \nabla \tau))^2}{|H||H_0|}] $.
\begin{lemma}
\begin{align*}
& f [V ^2 + V ^4 |\nabla \tau|^2 +  \frac{( div (V ^2 \nabla \tau))^2}{|H||H_0|}]  \\
= &  A r(h_0^{(1)} - h^{(1)})  +  r^2 [A(h_0^{(2)} - h^{(2)}) +\frac{W_0 B_i \tilde{X}^i}{A}]+ \\
& A r^3 \Big [(h_0^{(3)} - h^{(3)})  +  \frac{(h_0^{(2)} - h^{(2)})B_i \tilde{X}^i}{A^2}+ \frac{(W_0) [ g_1 + \frac{g_2}{4} -\frac{3W_0}{2}C_iC_j\tilde{X}^i \tilde{X}^j   ] }{2A^2} \Big ]+O(r^4).
\end{align*}
\end{lemma}

\begin{proof}
From Lemma \ref{lemma_7_1}, we have
\[
\begin{split}
    &V ^2 + V ^4 |\nabla \tau|^2 +\frac{(div V ^2 \nabla \tau)^2}{|H_0|^2} \\
=  &
A^2 +2AB_i \tilde{X}^i r \\
  &+ r^2  \left [ (B_i \tilde{X}^i)^2 +A^2+(C_i \tilde{X}^i)^2 + g_1+ \frac{g_2}{4}  - h_0^{(1)} \sum_{ij} C_iC_j \tilde {X}^i \tilde {X}^j \right],
\end{split}
 \]
and thus
\[
\begin{split}
   &|H_0| \sqrt{V ^2+ V ^4 |\nabla \tau|^2 +\frac{(div V ^2 \nabla \tau)^2}{|H_0|^2}}  \\
=&A \Big [ \frac{2}{r} + \frac{2AB_i \tilde{X}^i}{A^2} + r(h_0^{(1)} + \frac{ A^2+(C_i \tilde X^i)^2 +g_1+\frac{g_2}{4}  - h_0^{(1)} \sum_{ij } C_iC_j \tilde {X}^i \tilde {X}^j }{A^2} ) \Big ] + O(r^2).
\end{split}
\]
We can compute $|H| \sqrt{1+ |\nabla \tau|^2 +\frac{(\Delta \tau)^2}{|H|^2}} $ and $ 1+ |\nabla \tau|^2 +\frac{(\Delta \tau)^2}{|H_0||H|} $ similarly. As a result, $f (1 + |\nabla \tau|^2 +  \frac{(\Delta \tau)^2}{|H||H_0|})$ is equal to
\begin{align*}
  &(\frac{4}{r}  + (h_0^{(1)} + h^{(1)}) r) [(h_0^{(1)} - h^{(1)}) r +(h_0^{(2)} - h^{(2)}) r^2 +(h_0^{(3)} - h^{(3)}) r^3 ] \times  \\
&\frac{A^2 + B_i \tilde{X}^i r+ r^2 [(C_i \tilde{X}^i)^2+g_1 + \frac{g_2}{4}  - \frac{(h_0^{(1)} + h^{(1)})}{2} C_i C_j\tilde{X}^i \tilde{X}^j ] }{A\{  \frac{4}{r} + \frac{4B_i \tilde{X}^i}{A^2} +  r [ 2(C_i \tilde{X}^i)^2+h_0^{(1)} + h^{(1)} +  \frac{2g_1 + \frac{g_2}{2}  - 2(h_0^{(1)} + h^{(1)})C_iC_j\tilde{X}^i \tilde{X}^j  }{A^2} ]  \}  }  .\\
\end{align*} Finally we plug in $h_0^{(1)}=2W_0+1$ and $h^{(1)}=W_0+1$.
\end{proof}
\begin{lemma}\label{lemmaenergy}

\[\begin{split}
  &\lim_{r\rightarrow 0} r^{-5} \int _{\Sigma_r} f [V ^2 + V ^4 |\nabla \tau|^2 +  \frac{( div (V ^2 \nabla \tau))^2}{|H||H_0|}] \,\, d \Sigma_r\\
=&A \int_{S^2}  (h_0^{(3)} - h^{(3)}) dS^2   -\frac{3C_iC_j}{4 A} \int_{S^2} W_0^2\tilde {X}^i\tilde {X}^j dS^2\\
  &+ \frac{C_iC_j}{2A} \int_{S^2}(W_0)[R_{ij}+2\tilde{\nabla}\tilde {X}^i \cdot\tilde{\nabla}(Y_j^{(3)} + P_j)+\tilde{X^i}(S_j-\tilde{\Delta}Y_j^{(3)} +12P_j)]dS^2 \\
\end{split}\]
\end{lemma}
\begin{proof}
For the volume form, we have $ d\Sigma_r = r^2 dS^2 + O(r^5) $ from the expansion of metric in  Lemma \ref{non_physical_data}. As a result, it suffices to use $r^2 dS^2 $ for the volume form.

For the mean curvature in AdS space,
\[  |H_0| = 2 \sqrt{K+1} +O(r^3)  \]
since $Y_0 = O(r^3)$.
Using the Gauss equation of the surface in $N$, we conclude that
\[   |H_0| - |H| =  \rho r + D \rho r^2  + O(r^3)\]
That is
\[
\begin{split}
(h_0^{(1)} - h^{(1)})  =&  \rho \\
(h_0^{(2)} - h^{(2)})  =& D\rho
\end{split}
\]
It follows that for $i=1,2$
\[  \int_{S^2}  (h_0^{(i)} - h^{(i)}) dS^2 = 0.\]
Moreover, for parity reason, it is easy to see that
\[  \int_{S^2} W_0 \tilde X^i dS^2 =0. \]
Hence, we are left with only the $O(r^5)$ terms. We compute
\[\begin{split}  \int_{S^2} W_0 (  g_1 + \frac{g_2}{4}) dS^2
 = & C_iC_j\int_{S^2}W_0[R_{ij}+2\tilde{\nabla}\tilde {X}^i  \tilde \nabla (Y_j^{(3)}+P_j)+\tilde {X}^i(S_j-\tilde{\Delta}Y_j^{(3)}+12P_j)]dS^2\\
 &- AC_j\int_{S^2} W_0(\frac{2}{3}\tilde{\nabla} \tilde {X}^i\tilde{\nabla}W_0+2\tilde {X}^iW_0) dS^2.\end{split}\]
Due to parity, we have
\[
\int_{S^2} W_0(\frac{2}{3}\tilde{\nabla} \tilde {X}^i\tilde{\nabla}W_0+2\tilde {X}^iW_0) dS^2= 0.
\]
Finally, from Lemma \ref{D_divergence}, we have
\[
\int_{S^2} D\rho \tilde X^i  dS^2=   \frac{1}{20}   \int_{S^2}\tilde \nabla^a \tilde\nabla^b D\alpha_{ab} \tilde X^i dS^2=- \frac{1}{20}  \int_{S^2}D\alpha_{ab} \tilde \sigma^{ab} \tilde X^i dS^2=  0.
\]
Namely, 
\[
\int_{S^2} (h_0^{(2)} - h^{(2)}) \tilde X^i  dS^2 =0.
\]
The lemma follows from Lemma 9.2.
\end{proof}
\subsection{Computation of $\int  (h_0^{(3)} - h^{(3)})$}
Suppose $Y$ is the isometric embedding of $\sigma$ into the AdS space such that
\[
\begin{split}
Y_0  = & \sum_{i=3}^{\infty}Y_0^{(i)}r^i \\
Y_k = &  r \tilde X^k+ \sum_{i=3}^{\infty} Y_k^{(i)}r^i  \\
Y_4 =  & \sqrt{1 + \sum_i Y_i^2 - Y_0^2}
\end{split}
\]
where $Y_0^{(3)}$ and $Y_i^{(3)}$ are given by Lemma \ref{yi3} and Lemma \ref{y03}, respectively.

Let $Y'$ be the isometric embedding of $\sigma$ into the hyperbolic space in the AdS space where
\[(Y_0)'=0\]
\[(Y_i)'=r\tilde {X}^i+r^3 Y_i^{'(3)}+r^4 Y_i^{'(4)}+r^5 Y_i^{'(5)} +O(r^6).\]

Let $A'$ be the second fundamental form the embedding $Y'$ in the hyperbolic space and $\AA'$ be its traceless part.
\[
\AA'_{ab}=r^3\AA_{ab}^{'(3)}+O(r^4)\\
\]

Suppose the Gauss curvature $K$ of $\sigma$ has the following expansion:
\begin{equation}\label{k_exp} 2\sqrt{K}=\frac{2}{r} + k^{(1)} r +  k^{(2)} r^2 + k^{(3)} r^3 +O(r^4).\end{equation}
We have
\begin{proposition} \label{pe2}
The integral $\int_{S^2}  (h_0^{(3)} - h^{(3)}) dS^2$ can be written as follows:
\[\begin{split}\int_{S^2}  (h_0^{(3)} - h^{(3)}) dS^2= &   \frac{1}{2}\int_{S^2}|\AA^{'(3)}|^2_{\tilde{\sigma}} dS^2+\int (k^{(3)} - h^{(3)} - \frac{1}{4}) dS^2-\frac{2}{3}\int_{S^2} W_0^2 dS^2 \\
&-30\frac{C_iC_j}{A^2}\int_{S^2} P_i P_j dS^2.\end{split}\]
\end{proposition}
\begin{proof}
We first rewrite
\[  \int_{\Sigma_r} (|H_0| - |H|) d \Sigma_r = \int_{\Sigma_r}(|H_0| - 2\sqrt{K+1} )d \Sigma_r+ \int_{\Sigma _r} (2 \sqrt{K+1} -|H|)d \Sigma_r .\]
We have
\[  \int_{\Sigma_r} (2 \sqrt{K+1} - |H| ) d \Sigma_r=  r^5 \int_{S^2} (k^{(3)}-h^{(3)}- \frac{1}{4})dS^2 + O(r^6). \]
To evaluate $ \int_{\Sigma_r}(|H_0| - 2\sqrt{K+1} )d \Sigma_r$, recall that  $|H_0|^2$ is given by
\begin{equation}\label{h_0}|H_0|^2=-(\Delta Y_0)^2-(\Delta Y_4)^2+ \sum_{i=1}^3 (\Delta  Y_i)^2 + 4 .\end{equation}
since the AdS space is an umbilical hypersurface in $\R^{3,2}$.

Let $H_0'$  be the mean curvature of $Y'$.
Similarly, $| H_0'|$ is given by
\begin{equation}\label{h'} | H_0'|^2  = -(\Delta Y'_4)^2+ \sum_{i=1}^3 (\Delta  (Y_i)')^2 + 4 . \end{equation}
The Gauss equation of $Y'(\Sigma)$ in the hyperbolic space reads \begin{equation}\label{gauss}  4 K = (H_0')^2 - 4 - 2 |\AA'|^2.\end{equation}

We compute from \eqref{h_0}, \eqref{gauss}, and \eqref{h'} that \[|H_0|^2-4K-4=2|\AA'|^2-(\Delta Y_0)^2 -(\Delta Y_4)^2 + (\Delta Y'_4)^2+\sum_{i=1}^3(\Delta Y_i)^2-\sum_{i=1}^3(\Delta Y_i')^2,\]
where \[\Delta Y_0=\Delta(r^3 Y_0^{(3)}+O(r^4))=r\tilde{\Delta} Y_0^{(3)}+O(r^2)\]
$\Delta Y_4 = O(r^2)$, $\Delta Y'_4 = O(r^2)$ and
\[\begin{split}\sum_{i=1}^3(\Delta Y_i)^2-\sum_{i=1}^3(\Delta (Y_i)')^2
=&\sum_{i=1}^3\Delta (Y_i-Y_i') \Delta(Y_i+Y_i')\\
=&-4r^2\tilde{X}^i\tilde{\Delta}(Y_i^{(5)}-Y_i^{'(5)})+O(r^3).\end{split}\]
As a result, we have
\[
\begin{split}
   &\int_{S^2}  (h_0^{(3)} - h^{(3)}) dS^2 \\
= & \frac{1}{2}\int_{S^2} |\AA^{'(3)}|^2_{\tilde{\sigma}} dS^2-\frac{1}{4}\int_{S^2}  (\tilde{\Delta}Y_0^{(3)})^2 dS^2-\int_{S^2}  \tilde X^i \tilde{\Delta}(Y_i^{(5)}-Y_i^{'(5)}) dS^2+\int_{S^2}   (k^{(3)}- \frac{1}{4} - h^{(3)}) dS^2.
\end{split}
\]
To evaluate the second last terms, we need
\begin{lemma}\label{isom_higher}
If we choose $Y_i^{(3)}=Y_i^{'(3)}$ and $Y_i^{(4)}=Y_i^{'(4)}$, then $Y_i^{(5)}$ and $Y_i^{'(5)}$ are related by
\[2\tilde{\nabla} \tilde{X}^i\cdot \tilde{\nabla} (Y_i^{(5)}-Y_i^{'(5)})=|\tilde{\nabla} Y_0^{(3)}|^2.\]
\end{lemma}

\begin{proof}
This follows directly from the expansion of the metric and the isometric embedding equation.
\end{proof}
The proposition now follows from the expression of $Y_0^{(3)}$ in Lemma \ref{y03}.
\end{proof}
The traceless part $\AA_{ab}$ of $h_{ab}$ has the following expansion
\[\AA_{ab}=r^3 \AA_{ab}^{(3)}+O(r^4).\]
\begin{lemma}
\[\AA^{(3)}_{ab}=(\tilde {X}^i_a\tilde {X}^j_b+\tilde {X}^i_b \tilde {X}^j_a)(-\frac{1}{4}W_0 \delta_{ij}-\frac{1}{2}W_{0i0j}). \]
\end{lemma}
\begin{remark}
The proof is the same as Lemma 7.5 of \cite{Chen-Wang-Yau2} by embedding everything in $\R^{3,2}$ as above.
\end{remark}
As result, from Lemma 7.6 of \cite{Chen-Wang-Yau2}, we conclude again that
\begin{lemma} \label{lemmaAA}
\[ \int_{S^2} |\AA^{(3)}|^2_{\tilde{\sigma}} dS^2 =3\int_{S^2} W_0^2 dS^2. \]
\end{lemma}
\subsubsection{Computing $\int (k^{(3)} - \frac{1}{4} - h^{(3)}) dS^2$}
\begin{lemma}\label{lemma7.7}
\begin{equation}
 \int _{S^2} (k^{(3)} - \frac{1}{4}- h^{(3)}) dS^2 =-\frac{3}{4}\int_{S^2} W_0^2  d S^2-\frac{1}{60} \int _{S^2} |\alpha|^2 dS^2 + \frac{11}{45} \int _{S^2} |\beta|^2 dS^2.\end{equation}
\end{lemma}
\begin{proof}
First, we compute $\int k^{(3)}  dS^2$. From \eqref{k_exp}, we have
\[ K= \frac{1}{r^2} + k^{(1)} + k^{(2)}r  + [k^{(3)} + \frac{(k^{(1)})^2}{4}] r^2  + O(r^3).\]
We also have
\[ d \Sigma_r  = (r^2   - \frac{1}{180} r^6 |\alpha|^2) dS^2+O(r^7)\]
from the expansion of $\sigma^{ab} l_{ab}$ in Lemma \ref{data}. By the Gauss--Bonnet theorem $ \int_{\Sigma_r} K d \Sigma_r = 4 \pi$.
Collecting the $O(r^4)$ terms from the left hand side, we have
\[ \int _{S^2} k^{(3)} + \frac{(k^{(1)})^2}{4}  d S^2 = \frac{1}{180} \int_{S^2}  |\alpha|^2 dS^2. \]
Furthermore, $k^{(1)} = 2 W_0$. Hence
\[ \int _{S^2}  k^{(3)} dS^2 =-  \int _{S^2} W_0^2  d S^2 + \frac{1}{180} \int  _{S^2} |\alpha|^2 dS^2. \]
For $h^{(3)}$, we have
\[ h^{(3)} = (\sigma^{ab} n_{ab})^{(3)}- \frac{1}{90} |\alpha|^2 -  \frac{( \sigma^{ab} n_{ab}^{(1)} )^2}{4}.
 \]
Using Lemma \ref{data} and Lemma \ref{D_divergence}, we conclude
\begin{equation}
\begin{split}
 & \int_{S^2} (k^{(3)} - \frac{1}{4} - h^{(3)}) dS^2 \\
 =& -\frac{3}{4}\int_{S^2} W_0^2  d S^2 + \frac{1}{60} \int_{S^2} |\alpha|^2 dS^2- \int _{S^2}(\sigma^{ab} n_{ab})^{(3)} dS^2  + \int_{S^2}\frac{ W_0}{2} dS^2  \\
=&-\frac{3}{4}\int_{S^2} W_0^2  d S^2  -\frac{1}{60} \int _{S^2} |\alpha|^2 dS^2 + \frac{11}{45} \int _{S^2} |\beta|^2 dS^2.
\end{split}
\end{equation}
\end{proof}

\section{Computing the reference Hamiltonian}\label{sec_ref_ham}
In this section, we compute the limit of the second integral in equation \eqref{energy_expression}:
\[ \int_{\Sigma_r}  - \alpha_{H_0} (V ^2 \nabla \tau) d \Sigma_r.\]
First we prove the following lemma about $\alpha_{H_0} $.
\begin{lemma} Let $Y$ be an isometric embedding of $\Sigma_r$ into the AdS sapce such that 
\[
Y^0 = r^3 Y_0^{(3)} +  r^4 Y_0^{(4)} + O(r^5).
\]
We have
\begin{equation}\label{mean_gauge_H_0_two_order}
(\alpha_{H_0} )_a =  \frac{ r^2}{2} \tilde \nabla_a (\tilde \Delta + 2 )Y_0^{(3)}  +   \frac{ r^3}{2} \tilde \nabla_a (\tilde \Delta + 2 )Y_0^{(4)}  + O(r^4)
\end{equation}
and
\begin{equation}\label{div_mean_gauge_H_0_three_order}
\begin{split}div _{\sigma}\alpha_{H_0}=& \frac{1}{2} (r^{-2} \Delta)(r^{-2} \Delta +2) (Y_0^{(3)} + r Y_0^{(4)} +r^2 Y_0^{(5)})  - \frac{r^2}{2} ( \tilde \Delta (\tilde \Delta -1 )Y^{(3)}_0)+\\
&r^2[\tilde{\sigma}^{ac} \tilde{\sigma}^{bd} \AA^{'(3)}_{cd}  \tilde \nabla _b  \tilde \nabla _a Y_0^{(3)} + Y_0 \tilde \Delta Y_0^{(3)}
 + 2 \tilde \nabla W_0 \tilde \nabla Y_0^{(3)} -\frac{1}{2} \tilde \Delta ( W_0 \tilde \Delta Y^{(3)}_0)] +O(r^3) .
\end{split}
\end{equation}
\end{lemma}
\begin{proof} 
We may assume, for this lemma alone that $T_0 = \frac{\partial}{\partial_t}$ since the qualities depends on $Y$ but not on $T_0$.
\eqref{mean_gauge_H_0_two_order} follows from the formula in the proof of Theorem 6.2 of \cite{Chen-Wang-Yau3} by treating the image of the isometric embedding as a small perturbation of the
isometric embedding into the hyperbolic space. For \eqref{div_mean_gauge_H_0_three_order}, we use Theorem 5.2 of \cite{Chen-Wang-Yau3} that a surface in the AdS space is always a critical point of the quasi-local energy with respect to other isometric embeddings into the AdS space. That is, we consider a family of isometric embedding $Y(s)$ such that $Y(0)=Y$, we have
\[  \frac{d}{ds}|_{s=0} E(\Sigma, Y(s), \frac{\partial}{\partial t})= 0. \]
As a result, if we consider
\[\frak{H}_1= \int V \widehat H d \widehat \Sigma \]
and
\[\begin{split} \frak{H}_2&= \int  \Big [ \sqrt{(1+V ^2| \nabla \tau|^2) |H_0|^2  V ^2 + div(V ^2 \nabla \tau)^2 } \\
&-   div(V ^2 \nabla \tau)  \sinh^{-1} \frac{ div(V ^2 \nabla \tau) }{V |H_0|\sqrt{1+V ^2| \nabla \tau|^2} }   - V ^2 \alpha_{H_0} (\nabla \tau)  \Big ] d \Sigma\end{split}.\]
Then
\[  \frac{d}{ds}|_{s=0} \frak{H}_1=\frac{d}{ds}|_{s=0}\frak{H}_2,\]
where for the variation of $\frak{H}_2$, it is understood that $H_0$ and $\alpha_{H_0}$ are fixed at their values at the initial surface $Y(0)$. On $Y(0)$, we have
\[
V= \sqrt{1+r^2} + O(r^4).
\]
We consider a family of isometric embedding such that
\[  \frac{d}{ds}|_{s=0}  Y^0 = O(r^3)\]
and
\[  Z^i=  \frac{d}{ds}|_{s=0}  Y^i = O(r^5)\]
We conclude that
\[   \frac{d}{ds}|_{s=0}  Y^4 = O(r^6)\]
and
\[   \frac{d}{ds}|_{s=0}  V = O(r^6)\]
In terms of the static coordinate, we have
\[  y^0 = \sqrt{1+r^2} t  + O(r^6) \]
and
\[  Y^0 = \sqrt{1+r^2} \tau  + O(r^6).  \]
Let
\[
\delta \tau =  \frac{d}{ds}|_{s=0}  \tau
\]
\[ (\delta \hat \sigma)_{ab}  = (1+r^2) (\partial_a \tau \partial_b \delta \tau + \partial_b \tau \partial_a  \delta \tau ) + O(r^{10})   \]
and $Z^i$ satisfies the linearized isometric embedding equation
\[    \sum_i \partial_a Y^i \partial_b Z^i +\partial_b Y^i \partial_a Z^i = (1+r^2) (\partial_a \tau \partial_b \delta \tau + \partial_b \tau \partial_a  \delta \tau ) + O(r^{10}) \]
Decomposing $Z$ into the tangential part $P_a$ and normal part $ f e_3$ to the surface $\widehat \Sigma$
From equation (5.6) of \cite{Chen-Wang-Yau3}, we have
\begin{equation}
\delta \widehat H +\frac{1}{2} \hat h^{ab}  (\delta \hat \sigma)_{ab} =- \widehat \Delta f +2 f +   \hat \nabla^b( P^c \hat h_{cb}).
\end{equation}
As a result, we have
\[  \delta (V \widehat H) = \sqrt{1+r^2} (-\frac{1}{2} \hat h^{ab}  (\delta \hat \sigma)_{ab} - \widehat \Delta f +2 f +   \hat \nabla^b( P^c \hat h_{cb})) + 2 f + O(r^6) \]
and
\[
\begin{split}
   &\frac{d}{ds}|_{s=0} \frak{H}_1\\
= &\int  \sqrt{1+r^2} (-\frac{1}{2} \hat h^{ab}  (\delta \hat \sigma)_{ab} - \widehat \Delta f +2 f +   \hat \nabla^b( P^c \hat h_{cb})) + 2 f  +  \frac{1}{2} \hat \sigma^{ab}   (\delta \hat \sigma)_{ab}   \sqrt{1+r^2}   \widehat H   d \hat \Sigma +O(r^8)\\
= & \int  \sqrt{1+r^2}   \frac{1}{2} (\widehat H \hat \sigma^{ab} - \hat h^{ab})(\delta \hat \sigma)_{ab}  + 4f d \hat \Sigma +O(r^8)\\
\end{split}
 \]
Rewriting the linearized isometric embedding equation in terms of $P^a$ and $f$, we have
\[  \hat \nabla_a P_b + \hat \nabla_b P_a + 2f \hat h_{ab} = (\delta \hat \sigma)_{ab} \]
Taking the trace and integrating, we get
\[  \int 2 f \widehat H  d \hat \Sigma  = \int   \hat \sigma^{ab}  (\delta \hat \sigma)_{ab}   d \hat \Sigma   \]
It follows that
\[  \int 4f  d \hat \Sigma  = \int r   \hat \sigma^{ab}  (\delta \hat \sigma)_{ab}   d \hat \Sigma  + O(r^8) \]
and
\[
\begin{split}
\frac{d}{ds}|_{s=0} \frak{H}_1 =  & \int  \sqrt{1+r^2}   \frac{1}{2} (\widehat H \hat \sigma^{ab} - \hat h^{ab})(\delta \hat \sigma)_{ab}  + r   \hat \sigma^{ab}  (\delta \hat \sigma)_{ab}  d \hat \Sigma +O(r^8) \\
= &    \int  \delta \tau (1+r^2)^{\frac{3}{2}}  \hat \nabla _a \hat \nabla_b \tau  \frac{1}{2} (\widehat H \hat \sigma^{ab} - \hat h^{ab} + r   \hat \sigma^{ab}  )  d \hat \Sigma  +O(r^8)  \\
= &    \int   \delta \tau  (1+r^2)  \nabla _a  \nabla_b Y^0 \frac{1}{2} (\widehat H \hat \sigma^{ab} - \hat h^{ab} + r   \hat \sigma^{ab}  ) d \hat \Sigma   +O(r^8).
\end{split}
\]
On the other hand, direct computation (see Theorem 5.2 of \cite{Chen-Wang-Yau3}) gives that
\[
   \frac{d}{ds}|_{s=0} \frak{H}_2
= \int  (1+r^2)  \delta \tau  \left [   div (|H_0| \nabla Y^0) + \Delta \frac{\Delta Y^0}{|H_0|}    -div \alpha_{H_0}  \right].
 \]
We conclude that
\[
div \alpha_{H_0}  =  \Delta \frac{\Delta Y^0}{|H_0|}  +div (|H_0| \nabla Y^0)     - \frac{1}{2} ( (\frac{1}{2} |H_0| +r)\sigma^{ab} - \AA^{ab}  )  )  \nabla _a  \nabla_b Y^0  + O(r^3).
\]
since
\[  \frac{d}{ds}|_{s=0} \frak{H}_1=\frac{d}{ds}|_{s=0}\frak{H}_2,\]
for any choice of $\delta \tau$
\end{proof}

\begin{proposition} \label{pv0}
\[    \lim_{r \to 0} -r^{-5 }\int_{\Sigma_r}  \alpha_{H_0} (V ^2 \nabla \tau)d \Sigma_r
=  \frac{4}{3} A \int _{S^2} W_0^2 dS^2 - 10  \frac{C_iC_j}{A}\int _{S^2}\tilde X^i W_0 P_j dS^2.
\]
\end{proposition}
\begin{proof}
We compute

\[
\begin{split}
    &  \int_{\Sigma_r } \alpha_{H_0} (V ^2 \nabla \tau)d \Sigma_r   \\
 = &   \int_{\Sigma_r } \alpha_{H_0} \left [ (A Y^0 + C_i Y^i)\nabla Y^4 + (AY^4 + B_i Y^i)\nabla Y^0 + (C_i Y^4 + B_i Y^0 + D_p  \epsilon_{pqi}  Y^q)\nabla Y^i \right] d \Sigma_r\\
 = &   \int_{\Sigma_r } \alpha_{H_0} \left [  A\nabla Y^0 + (C_i Y^4 + D_p  \epsilon_{pqi}  Y^q)\nabla Y^i \right] d \Sigma_r  + O(r^6)\\
 = &  \int_{\Sigma_r } \alpha_{H_0} \left [  A\nabla Y^0 + (C_i (1+ \frac{r^2}{2})  + D_p  \epsilon_{pqi}  Y^q)\nabla Y^i \right] d \Sigma_r  + O(r^6)\\
 = &  \int_{\Sigma_r } \alpha_{H_0} \left [  A\nabla Y^0 + C_i \nabla Y^i \right] d \Sigma_r + O(r^6).
 \end{split}
\]
In the last equation, we used \eqref{mean_gauge_H_0_two_order} and that $Y_0^{(3)}$ is perpendicular to $\tilde X^i$. Integrating by parts, we get
\[   - \int_{\Sigma_r } \alpha_{H_0} (V ^2 \nabla \tau)d \Sigma_r =   \int_{\Sigma_r }  (AY^0 + C_i Y^i) div  \alpha_{H_0}+ O(r^6) \]
The proposition now follows from Proposition 8.1 of \cite{Chen-Wang-Yau3}.
\end{proof}
\section{Computing the Physical Hamiltonian}\label{sec_phy_ham}

In this section, we compute the limit of the third integral in equation \eqref{energy_expression}:
\[ -\int_{\Sigma_r} \alpha_{H} (V ^2 \nabla \tau)d \Sigma_r.\]
\begin{proposition}\label{pvr}
\begin{align*}
    \lim_{r \to \infty} -r^{-5 } \int_{\Sigma_r } \alpha_{H} (V ^2 \nabla \tau)d \Sigma_r
  =  \int_{S^2} \left[\frac{4A}{3}W_0^2 +\frac{2}{3} C_iW_iW_0- (C_i \tilde {X}^i) |\beta|^2 \right]dS^2.
\end{align*}
\end{proposition}
\begin{proof}
We compute

\[
\begin{split}
    &  \int_{\Sigma_r } \alpha_{H} (V ^2 \nabla \tau)d \Sigma_r   \\
 = &   \int_{\Sigma_r } (\alpha_H) \left [  A\nabla Y^0 + (C_i Y^4 + D_p  \epsilon_{pqi}  Y^q)\nabla Y^i \right] d \Sigma_r  + O(r^6)\\
 = &  \int_{\Sigma_r } (\alpha_H) \left [  A\nabla Y^0 + (C_i (1+ \frac{r^2}{2})  + D_p  \epsilon_{pqi}  Y^q)\nabla Y^i \right] d \Sigma_r  + O(r^6)\\
 \end{split}
\]
From Proposition 9.1 of \cite{Chen-Wang-Yau2}, we have
\[  \int_{\Sigma_r } (\alpha_H) \left [  A\nabla Y^0 + C_i  \nabla Y^i \right] d \Sigma_r   = r^5     \int_{S^2} \left[\frac{4A}{3}W_0^2 +\frac{2}{3}C_i W_iW_0- (C_i \tilde {X}^i) |\beta|^2 \right]dS^2
 + O(r^6)\]
As a result, it suffices to show that
\[
 \int_{\Sigma_r } (\alpha_H) \left [  (\frac{r^2C_i}{2} + D_p  \epsilon_{pqi}  Y^q)\nabla Y^i \right] d \Sigma_r  = O(r^6)
\]
or simply that
\[
\begin{split}
\int_{S^2}  \alpha_H^{(2)}( \tilde \nabla \tilde X^i) = & 0\\
\int_{S^2}  \alpha_H^{(2)}(  \epsilon_{pqi} \tilde X^q \tilde \nabla \tilde X^i) = & 0\\
\int_{S^2}  \alpha_H^{(3)}(  \epsilon_{pqi} \tilde X^q \tilde \nabla \tilde X^i) = & 0.
\end{split}
\]
The first integral vanishes due to Theorem 6.1 and that the spacetime is vacuum. For the other two integrals, we recall that
\begin{equation} \eta_a = \frac{r^2}{3} \beta_a +\frac{r^3}{4}   D\beta_a +r^4 [\frac{1}{10}   D^2 \beta_a  -\frac{1}{45} \alpha_{ab} \beta^b ] +O(r^5), \end{equation}
$\alpha_H$ is the same as $\eta$ up to a gradient vector field and the rotation Killing field $ \epsilon_{pqi} \tilde X^q \tilde \nabla \tilde X^i$ is divergence free. Hence, it suffice to show that
\[
\begin{split}
\int_{S^2} \beta (  \epsilon_{pqi} \tilde X^q \tilde \nabla \tilde X^i) = & 0\\
\int_{S^2} D\beta (  \epsilon_{pqi} \tilde X^q \tilde \nabla \tilde X^i) = & 0.
\end{split}
\]
From Lemma 4.7 and 4.9, we have
\[
\begin{split}
4 \int_{S^2} \beta (  \epsilon_{pqi} \tilde X^q \tilde \nabla \tilde X^i) = &\int_{S^2} \tilde \nabla^a \alpha_{ab}  (  \epsilon_{pqi} \tilde X^q \tilde \nabla \tilde X^i)  = 0 \\
5 \int_{S^2} D\beta (  \epsilon_{pqi} \tilde X^q \tilde \nabla \tilde X^i) =& \int_{S^2} \tilde \nabla^a D \alpha_{ab}  (  \epsilon_{pqi} \tilde X^q \tilde \nabla \tilde X^i)  = 0.
\end{split}
\]
\end{proof}
\section{Evaluating the energy}\label{sec_eva_energy}
From Lemma \ref{lemmaenergy}, Proposition \ref{pe2}, \ref{pv0} and \ref{pvr} and Section 10 of \cite{Chen-Wang-Yau2}, we conclude immediately that
\[
\begin{split}
     & \lim_{r \to 0} r^{-5} E(\Sigma_r, Y_r(T_0),T_0)= \frac{1}{90}  \Big[  Q(e_0,e_0,e_0,Ae_0 + C_i e_i)+ \frac{\sum_{m,n} \bar W_{0m0n}^2}{2A}\Big].
\end{split}
\]
To minimize $E(\Sigma_r, Y_r(T_0),T_0)$ among choices of $T_0$, let
\[ E(\Sigma_r, Y_r(T_0),T_0) =  E(\Sigma,Y(T_0),T_0)^{(5)} r^5 + O(r^6).\]
We consider the following vector
\[ U=(\frac{1}{2}\sum \bar{W}_{0kmn}^2+  \sum \bar{W}_{0m0n}^2, 2\sum \bar{W}_{0m0n}\bar{W}_{0min}).\] 
$U$ is future directed non-spacelike. It is timelike unless the Weyl curvature is of the form given by \cite[Lemma 11.2]{Chen-Wang-Yau1}.

It is easy to see that the energy functional $E(\Sigma,Y(T_0),T_0)^{(5)}$ is non-negative. Moreover, it is positive and proper when $U$ is timelike.
Hence, when $U$ is timelike, there is at least one choice of $T_0=(A,C_i)$ which minimizes $E(\Sigma,Y(T_0),T_0)^{(5)}$. We show that under the same condition, the minimizer is unique. 
\begin{lemma}\label{minimize_direction}
If  $V$ is timelike, there is a unique $(A, C_i)$ that minimizes $E(\Sigma,X(T_0),T_0)^{(5)}$.
\end{lemma}
\begin{proof}
From the remark after Proposition \ref{observer_constraint}, we have
\[  A \ge \sqrt{1+ |\vec{C}|^2} \ge 1. \]
Moreover, 
\[ E(\Sigma,Y(T_0),T_0)^{(5)} =  \frac{1}{90}  \Big[  Q(e_0,e_0,e_0,Ae_0 + C_i e_i)+ \frac{\sum_{m,n} \bar W_{0m0n}^2}{2A}\Big]  \]
where
\[  Q(e_0,e_0,e_0,Ae_0 + C_i e_i)= (\frac{1}{2}\sum_{k,m,n} \bar W_{0kmn}^2+\sum_{m,n} \bar W_{0m0n}^2 )A+ 2\sum_{m,n,i} \bar W_{0m0n} \bar W_{0min} C_i. \] 
Hence, fixing $\vec{C}$, $E(\Sigma,Y(T_0),T_0)^{(5)}$  is increasing in $A$ and the minimum can only occur on the set of observers, $O$, such that  
\[  A = \sqrt{1+ |\vec{C}|^2}. \]
From the proof of \cite[Lemma 11.3]{Chen-Wang-Yau1}, $E(\Sigma,Y(T_0),T_0)^{(5)}$ is a strictly convex function of $(C_1,C_2,C_3)$ on $O$. This finishes the proof of the lemma \end{proof}
As a result, we have the following theorem for the small sphere limit of the quasi-local energy in vacuum spacetimes with reference in the AdS space:
\begin{theorem} \label{thm_small_vac}
Let $\Sigma_r$ be the family of spheres approaching $p$ constructed in Section 4. 
\begin{enumerate}
\item For each observer $T_0$ in the AdS space, there is a pair $(Y_r(T_0),T_0)$ solving the leading order term of the optimal embedding equation of $\Sigma_r$ (see Lemma \ref{yi3} and \ref{y03}). For this pair $(Y_r(T_0),T_0)$, we have 
\[  
\begin{split}
     & \lim_{r \to 0} r^{-5} E(\Sigma_r, Y_r(T_0),T_0)= \frac{1}{90}  \Big[  Q(e_0,e_0,e_0,A e_0 + C_i e_i)+ \frac{\sum_{m,n} \bar W_{0m0n}^2}{2A}\Big].
\end{split}
\]

\item Suppose $Q(e_0,e_0,e_0,\cdot)$ is dual to a timelike vector. Let $\mathcal P$ denote the set of $(Y,T_0)$ admitting a power series expansion given in equation \eqref{assume}. We have 
\[  
\begin{split}
    &\inf_{(Y,T_0) \in \mathcal P} \lim_{r \to 0} r^{-5} E(\Sigma_r, Y,T_0) 
 =  \inf_{(A,C_i) \in \mathbb H^3} \frac{1}{90}  \Big[ Q(e_0,e_0,e_0,A e_0 + C_i e_i) + \frac{\sum_{m,n} \bar W_{0m0n}^2}{2A}\Big].
\end{split}
\]
where $\mathbb H^3$ denotes the set of unit timelike future directed vector in $\R^{3,1}$. The infimum is achieved by a unique $(A,C_i) \in \mathbb H^3$.
\end{enumerate}
\end{theorem}

\end{document}